\documentclass[oneside,11pt]{amsart}
\usepackage{amsmath}
\usepackage{amsthm}
\usepackage{pinlabel}
\usepackage{epstopdf}
\usepackage{mathtools} 
\usepackage{tikz}
\usepackage{graphicx}
\usepackage{pinlabel}
\usetikzlibrary{cd}

\usepackage[T1]{fontenc}

\usepackage[small]{caption}

\usepackage{amsfonts}
\usepackage{amssymb}
\usepackage{amsxtra}
\usepackage{thmtools} 
\usepackage{thm-restate}
\usepackage{eucal}

\usepackage[arrow, tips, matrix]{xy}
\SelectTips{cm}{12}
\CompileMatrices

\newtheorem{thm}{Theorem}[section]

\newtheorem{prop}[thm]{Proposition}
\newtheorem{lem}[thm]{Lemma}
\newtheorem{cor}[thm]{Corollary}

\theoremstyle{definition}
\newtheorem{defn}[thm]{Definition}
\newtheorem{rem}[thm]{Remark}

\renewcommand{\bar}[1]{\overline{#1}}
\newcommand{\boundary}{\partial}

\newcommand{\abs}[1]{\left|{#1}\right|}
\newcommand{\bigabs}[1]{\bigl| {#1} \bigr|}

\newcommand{\set}[2]{\{\,{#1} \mid {#2} \,\}}
\newcommand{\bigset}[2]{ \bigl\{ \, {#1} \bigm| {#2} \, \bigr\} }

\renewcommand{\setminus}{-}

\newcommand{\mct}{\mathcal{T}}

\newcommand{\field}[1]{\mathbb{#1}}
\newcommand{\Z}{\field{Z}}
\newcommand{\R}{\field{R}}

\newcommand{\Hyp}{\field{H}}

\newcommand{\GaP}{(\Gamma, \mathbb{P})}
\newcommand{\GaPp}{(\Gamma ', \mathbb{P}')}

\usepackage{bbm}
\newcommand{\id}{\mathbbm{1}}

\newcommand{\of}{\circ}

\DeclareMathOperator{\Isom}{Isom}
\DeclareMathOperator{\CAT}{CAT}

\DeclareMathOperator{\Stab}{Stab}

\DeclareMathOperator{\Cayley}{Cayley}

\usepackage{combelow}
\newcommand{\Drutu}{Dru{\cb{t}}u}

\usepackage{color}
\usepackage{pinlabel}

\usepackage{ifthen}

\newcommand{\showcomments}{yes}

\newsavebox{\commentbox}
{\ifthenelse{\equal{\showcomments}{yes}}%
{\footnotemark
    \begin{lrbox}{\commentbox}
    \begin{minipage}[t]{1.25in}\raggedright\sffamily\upshape\tiny
    \footnotemark[\arabic{footnote}]}
{\begin{lrbox}{\commentbox}}}%
{\ifthenelse{\equal{\showcomments}{yes}}%
{\end{minipage}\end{lrbox}\marginpar{\usebox{\commentbox}}}
{\end{lrbox}}}

\begin{document}

\title[Cusped spaces and quasi-isometries]{Cusped spaces and quasi-isometries of relatively hyperbolic groups}

\author{Brendan Burns Healy}
\address{Department of Mathematical Sciences\\
         University of Wisconsin--Milwaukee\\
         PO Box 413\\
         Milwaukee, WI 53211\\
	 USA}
\email{healyb@uwm.edu}

\author{G.~Christopher Hruska}
\address{Department of Mathematical Sciences\\
         University of Wisconsin--Milwaukee\\
         PO Box 413\\
         Milwaukee, WI 53211\\
	 USA}
\email{chruska@uwm.edu}

\begin{abstract}
A group $\Gamma$ with a family of subgroups $\mathbb{P}$ is relatively hyperbolic if $\Gamma$ admits a cusp-uniform action on a proper $\delta$--hyperbolic space. 
We show that any two such spaces for a given group pair are quasi-isometric, provided the spaces have ``constant horospherical distortion,'' a condition satisfied by Groves--Manning's cusped Cayley graph and by all negatively curved symmetric spaces.
Consequently the Bowditch boundary admits a canonical quasisymmetric structure, which coincides with the ``naturally occurring'' quasisymmetric structure of the symmetric space when considering lattices in rank one symmetric spaces.

We show that a group $\Gamma$ is a lattice in a negatively curved symmetric space $X$ if and only if a cusped space for $\Gamma$ is quasi-isometric to the symmetric space.
We also prove an ideal triangle characterization of the $\delta$--hyperbolic spaces with uniformly perfect boundary due to Meyer and Bourdon--Kleiner.
An appendix concerns the equivalence of several definitions of conical limit point found in the literature.
\end{abstract}



\date{\today}

\maketitle

\section{Introduction}
\label{sec:Introduction}

This article provides an introduction to relatively hyperbolic groups and metric structures on their Bowditch boundaries.
In order to endow the Bowditch boundary with a canonical metric structure, one needs the group to act appropriately on a $\delta$--hyperbolic space whose large-scale geometry (up to quasi-isometry) is canonical.
In the literature, several such model geometries have been used (see \cite{CannonCooper92,BowditchRelHyp,GrovesManning08DehnFilling}) and are known to be quasi-isometric (see \cite{GrovesManningSisto}).
The main goal of this paper is to isolate and abstract a common geometric feature of these models that leads to their quasi-isometric invariance: the notion of horoballs with constant horospherical distortion.
As a consequence of this study, we observe that horoballs in all negatively curved symmetric spaces also have constant horospherical distortion, a fact that is implicit in work of Schwartz \cite{Schwartz95} and made explicit here.
In support of this investigation, we provide detailed proofs of several folk results that are well-known to the experts, but whose proofs we did not find in the literature.

The notion of a relatively hyperbolic group extends Gromov's original definition of a hyperbolic group.
Hyperbolic groups have a geometry that generalizes certain coarse features of uniform lattices in symmetric spaces with negative sectional curvature.
In the study of rank one uniform lattices, the analytic study of the sphere at infinity plays a prominent role, for instance in proofs of Mostow Rigidity and many subsequent quasi-isometric rigidity theorems (see \cite{KleinerLeeb01} for a survey).

Each $\delta$--hyperbolic space $X$ has a Gromov boundary at infinity $\boundary X$, which generalizes the sphere at infinity and which admits a canonical conformal gauge, \emph{i.e.}, a choice of metric that is well-defined up to quasisymmetric homeomorphism.
A group is hyperbolic if it acts  properly, cocompactly, and isometrically on a $\delta$--hyperbolic metric space.
A hyperbolic group $\Gamma$ acts on its Gromov boundary $\boundary \Gamma$, which is defined to be the Gromov boundary of any $\delta$--hyperbolic space on which $\Gamma$ acts properly, cocompactly, and isometrically. The Gromov boundary $\boundary \Gamma$ is well-defined in the following sense: if $\Gamma$ acts geometrically on two $\delta$--hyperbolic spaces $X$ and $Y$, there is an induced $\Gamma$--equivariant quasisymmetric homeomorphism $\boundary X \to \boundary Y$.
In the setting of hyperbolic groups, analytic study of the Gromov boundary leads to powerful theorems about large-scale geometry, including \cite{Bourdon97,BonkSchramm00,BonkKleiner05,Mackay_Random} and numerous others.

Relatively hyperbolic groups have a geometry that generalizes the nonuniform lattices in negatively curved symmetric spaces. Recall that if $X$ is a negatively curved symmetric space, a nonuniform lattice acts properly and cocompactly on the subspace of $X$ formed by removing a certain family of disjoint open horoballs
\cite{Borel_Arithmetic,GarlandRaghunathan70,Bowditch95}.
Such an action on a $\delta$--hyperbolic space is \emph{cusp uniform}.
A \emph{group pair} $(\Gamma,\mathbb{P})$ consists of a group $\Gamma$ together with a finite family of subgroups $\mathbb{P}$.
A group pair $\GaP$ is \emph{relatively hyperbolic} if $\Gamma$ admits a cusp-uniform action on a $\delta$--hyperbolic space $X$ such that every maximal parabolic subgroup is conjugate to a member of $\mathbb{P}$ (see \cite{Gromov87,BowditchRelHyp}). Such a space $X$ is a \emph{weak cusped space} for $\GaP$.

A relatively hyperbolic pair $\GaP$ has a Bowditch boundary $\boundary \GaP$ that is topologically well defined in the following sense: any two weak cusped spaces $X$ and $Y$ for a relatively hyperbolic pair $\GaP$ have Gromov boundaries that are $\Gamma$--equivariantly homeomorphic \cite{BowditchRelHyp,GerasimovPotyagailo16_Similar}.
However in general, there need not be a $\Gamma$--equivariant quasi-isometry between $X$ and $Y$ \cite{Healy}.
Thus in this setting the Bowditch boundary does not appear to have a natural choice of metric structure.

To study quasi-isometries and the associated analysis on the boundary of relatively hyperbolic groups, one requires more control over the choice of weak cusped space.
We assume at this point that each peripheral subgroup $P \in \mathbb{P}$ is finitely generated, and thus has a natural word metric (up to bilipschitz equivalence).
In order to endow the Bowditch boundary with a suitable metric, Bowditch observes in \cite{Bowditch99Boundaries} that one may always choose a weak cusped space whose horoballs satisfy the following geometric criterion:
a horoball corresponding to a maximal parabolic subgroup $P$ has \emph{constant horospherical distortion} if the ``horospherical'' orbits $Px$ are exponentially distorted in $X$ with a lower and upper bound that are exponential functions with the same base (see Section~\ref{sec:HoroDistortion} for details).

This condition, inspired by the geometry of real hyperbolic space, is quite strong; in general for manifolds of pinched negative curvature, the horospherical distortion has upper and lower bounds that are exponential functions with different bases (see Heintze--Im~Hof \cite{HeintzeImHof77}).
Nevertheless, in this article we show that constant horospherical distortion is a natural geometric condition with strong quasi-isometric stability properties.  
A \emph{cusped space} is a weak cusped space whose given parabolic horoballs each have constant horospherical distortion.

The notion of constant horospherical distortion is the basis for Groves--Manning's combinatorial horoball construction (inspired by an earlier construction of Cannon--Cooper \cite{CannonCooper92}).
Attaching combinatorial horoballs to a Cayley graph produces a \emph{cusped Cayley graph}, which has become one of the standard geometric model spaces for relatively hyperbolic groups \cite{CannonCooper92,BowditchRelHyp,GrovesManning08DehnFilling}. The notion of constant horospherical distortion is flexible enough that it is also satisfied by the horoballs in any negatively curved symmetric space.
The following theorem summarizes the central examples and quasi-isometric properties of cusped spaces.

\begin{thm}[Cusped spaces]
\label{thm:introcuspedspaces}
Let $(\Gamma,\mathbb{P})$ be a relatively hyperbolic group pair such that $\Gamma$ is finitely generated.
   \begin{enumerate}
   \item
   \label{item:CanonCooper}
   The cusped Cayley graph of Groves--Manning is a cusped space for $(\Gamma,\mathbb{P})$.
   \item
   \label{item:SymmetricCusped}
   If $X$ is a negatively curved symmetric space and $(\Gamma,\mathbb{P})$ is a lattice with its standard relatively hyperbolic structure, then $X$ is a cusped space for $(\Gamma,\mathbb{P})$.
   \item
   \label{item:CanonicalQI}
   All cusped spaces for $(\Gamma,\mathbb{P})$ are equivariantly quasi-isometric.
   \item
   \label{item:RelQC}
   Let $(H,\mathbb{Q}) \le (\Gamma,\mathbb{P})$ be a finitely generated, relatively quasiconvex subgroup with its induced relatively hyperbolic structure.
   The inclusion $H \to \Gamma$ induces a quasi-isometric embedding of cusped spaces.
   \end{enumerate}
\end{thm}

We note that Groves and Manning use the term ``cusped space'' to refer to the cusped Cayley graph.
Since the definition of ``cusped space'' in this paper includes the cusped Cayley graph as a special case, there is no conflict of terminology between the two.

In principle, Theorem~\ref{thm:introcuspedspaces}(\ref{item:CanonCooper}) follows closely from the definition of the cusped Cayley graph (see Section~\ref{sec:HoroDistortion} for details).
Theorem~\ref{thm:introcuspedspaces}(\ref{item:SymmetricCusped}) is a consequence of large-scale properties of sub-Riemannian metrics on the nilpotent horospheres in negatively curved symmetric spaces and is implicit in work of Schwartz \cite{Schwartz95}.
Theorem~\ref{thm:introcuspedspaces}(\ref{item:RelQC}) follows from work of Agol--Groves--Manning and Manning--Mart\'{i}nez \cite{AgolGrovesManning_QCERF,ManningMartinez10}, who characterize relative quasiconvexity in terms of the cusped Cayley graph.

A key tool used in the proof of Theorem~\ref{thm:introcuspedspaces}(\ref{item:CanonicalQI}) is a ``cusp extension'' technique modeled on \cite[\S 4.2]{CannonCooper92} and \cite[\S 5]{Schwartz95}, which has become something of a standard method in geometric group theory (see \cite{BowditchRelHyp,Groff13,Durham_Augmented,GrovesManningSisto}). 
Theorem~\ref{thm:cuspextintro} extends and unifies these applications of the cusp-extension technique. A \emph{quasi-isometry of pairs} $(\Gamma,\mathbb{P}) \to (\Gamma',\mathbb{P}')$ is a quasi-isometry of groups that preserves the family of left cosets of peripheral subgroups up to finite distance (see Section~\ref{sec:CuspExtension}).
A theorem of Behrstock--\Drutu--Mosher states that if all peripheral subgroups of $\mathbb{P}$ and $\mathbb{P}'$ are non--relatively hyperbolic, then every quasi-isometry $\Gamma \to \Gamma'$ between relatively hyperbolic groups is a quasi-isometry of pairs \cite{BehrDruMosher}.

\begin{thm}[Cusp extension]
\label{thm:cuspextintro}
Let $\GaP$ and $\GaPp$  be relatively hyperbolic with cusped spaces $X$ and $X'$.
Any quasi-isometry of pairs $(\Gamma,\mathbb{P}) \rightarrow (\Gamma',\mathbb{P}')$ has an extension to a quasi-isometry of cusped spaces $X \to X'$. Furthermore, if $\GaP=\GaPp$, then the extension is roughly equivariant.
\end{thm}

A quasi-isometry is roughly equivariant if it commutes with the group action up to a uniformly bounded error (see Definition~\ref{def:roughequivariance}).

We use the fact that rank one symmetric spaces have constant horospherical distortion to show that one can recognize their nonuniform lattices from the large-scale geometry of their cupsed Cayley graphs. The following theorem generalizes a result of Cannon--Cooper \cite{CannonCooper92} from the setting of real hyperbolic $3$--space to an arbitrary negatively curved symmetric space.

\begin{thm}
\label{thm:IntroRecognizingLattices}
Let $\GaP$ be a relatively hyperbolic pair. Then the cusped Cayley graph $Y$ for $\Gamma$ is quasi-isometric to a negatively curved symmetric space $X$ if and only if there exists a finite normal subgroup $F \triangleleft \Gamma$ such that $\Gamma/F$ is isomorphic to a lattice in $\Isom(X)$
and $\mathbb{P}$ represents the finitely many conjugacy classes of preimages in $\Gamma$ of the maximal parabolic subgroups of the lattice.
\end{thm}

The proof of the above statement relies on applying classical rigidity theorems for rank one symmetric spaces as in \cite{KleinerLeeb01}. These theorems must be applied with care due to certain technical restrictions in the noncocompact case that are not present in the cocompact case (in particular, see \cite{Tukia86,Tukia94}).
The proof also depends on the fact, promoted by Bowditch, that a group $\Gamma$ of isometries of a rank one symmetric space is a lattice if and only if its action on the boundary sphere is a geometrically finite convergence group, which is a completely topological property
\cite{BeardonMaskit74,GehringMartin87,Bowditch95}.

The rich classical study of analysis on the boundary sphere of a rank one symmetric space has been applied in many ways to understand both uniform and nonuniform lattices.
In particular the rigidity theorems mentioned above are proved using analytic properties of the canonical conformal gauge on the boundary sphere of the symmetric space.
On the other hand, for general relatively hyperbolic groups, the cusped Cayley graph (and its variants in \cite{CannonCooper92,BowditchRelHyp}) has been proposed as a standard model for metrizing the Bowditch boundary (see, for example, \cite{GrovesManningSisto,MackaySisto_RelHyp}).

The following corollary, in essence, ensures that the classical analytical study of the boundary of symmetric spaces coincides with the analytic study of the Bowditch boundary determined by the cusped Cayley graph. 

\begin{cor}
\label{cor:intro}
Let $\Gamma$ be a nonuniform lattice in a rank one symmetric space $X$, and $Y$ its cusped Cayley graph with respect to the maximal parabolic subgroups. Then there exists a quasisymmetric homeomorphism between the Gromov boundaries $\partial X$ and $\partial Y$ equipped with any visual metric.
\end{cor}

In fact, the above corollary is a specific case of a more general result that all visual metrics of all cusped spaces for a relatively hyperbolic pair are quasisymmetric, thus allowing one to make analytical statements about the Bowditch boundary of any relatively hyperbolic group pair.

\begin{cor}[Conformal Gauge]
\label{cor:introconformalgauge}
Let $\GaP$ be a relatively hyperbolic pair with finitely generated peripheral subgroups, and let $X$ be  cusped space for $\GaP$. If $\mathcal{G}$ is the conformal gauge of $\partial X$, then whenever $Y$ is a cusped space for $\GaP$ such that $\partial Y$ is equipped with a visual metric, we have that $\partial Y \in \mathcal{G}$. \end{cor}

\subsection{Uniformly perfect boundaries}
\label{sec:Intro:UniformPerfect}

In the analysis of metric spaces, the uniformly perfect spaces are especially well behaved and share many features with connected spaces (see, for instance, \cite{Heinonen01}).
A metric space is \emph{uniformly perfect} if there is $0< \lambda<1$ such that each point $p$ is the limit of a sequence $(p_n)$ satisfying $\lambda^{n+1} \le d(p_n,p) \le \lambda^n$.
It is well-known that every quasi-isometry between $\delta$--hyperbolic spaces induces a quasisymmetric map of their boundaries.
Conversely, a quasisymmetric map between boundaries is induced by a quasi-isometry, provided that the boundaries in question are known to be uniformly perfect (see Section~\ref{sec:Quasisymmetric} for details).

Thus it is desirable to have a geometric criterion for determining whether the boundary of a $\delta$--hyperbolic space is uniformly perfect.
In his dissertation, Meyer gives a geometric criterion for uniform perfection in the setting of general $\delta$--hyperbolic spaces \cite{Meyer_UniformlyPerfect}.
Bourdon--Kleiner observe that this criterion reduces to the following particularly simple form when the hyperbolic space is proper \cite{BourdonKleiner15}.

\begin{prop}[Meyer, Bourdon--Kleiner]
\label{prop:IntroUniformlyPerfect}
Let $X$ be a proper, visual $\delta$--hyperbolic space such that $\boundary X$ contains at least two points.
The boundary $\boundary X$ is uniformly perfect if and only if there exists a constant $K<\infty$ such that each point $x\in X$ lies within distance at most $K$ from all three sides of some ideal triangle.
\end{prop}

A detailed proof of this result in the form stated above has not appeared in the literature, so a proof using Meyer's theorem has been included in Section~\ref{sec:UniformlyPerfect} below.

A $\delta$--hyperbolic space admitting a cocompact group action always has uniformly perfect boundary (see Section~\ref{sec:UniformlyPerfect}).
However, in the context of relatively hyperbolic pairs $\GaP$, a weak cusped space does not necessarily admit a cocompact group action \cite{DahmaniYaman_BoundedGeometry}.
Nevertheless, we show in Proposition~\ref{prop:CuspUniformlyPerfect} that every weak cusped space for the relatively hyperbolic $\GaP$ has uniformly perfect boundary.

Suppose $\GaP$ is a relatively hyperbolic group pair.
In the family of all weak cusped spaces for $\GaP$, the property of having constant horospherical distortion is an equivariant quasi-isometry invariant in the following sense. 

\begin{thm}
\label{thm:introconformalgaugeconverse}
Let $\GaP$ be a relatively hyperbolic pair with finitely generated peripheral subgroups. Suppose $X_1$ and $X_2$ are weak cusped spaces for $\GaP$.
If there exists an equivariant quasisymmetric homeomorphism $\partial X_1 \to \partial X_2$ and $X_1$ is a cusped space then $X_2$ is also a cusped space.
\end{thm}

We note that the proof of this theorem relies on the uniform perfection of boundaries of all weak cusped spaces.

\subsection{Organization of the sections}

Section~\ref{sec:Quasisymmetric} reviews some terminology and machinery for the study of $\delta$--hyperbolic spaces and their boundary. Additionally we recall the correspondence between quasi-isometries of hyperbolic spaces and quasisymmetric boundary maps, and we review the notion of a quasi-action.
Section~\ref{sec:ConvergenceGroup} gives background on discrete convergence group actions. In particular, we recall that the induced action of a discrete isometry group on the boundary of a hyperbolic space has the form of a convergence group action. We include information on how this boundary action encodes relevant data about the action on the space itself.
In Section~\ref{sec:RelHyp}, we provide several equivalent characterizations of relative hyperbolicity and discuss the Bowditch boundary of a relatively hyperbolic pair.

Section~\ref{sec:BoundedParabolic} contains geometric results about the horoballs induced by cusp-uniform actions. We introduce bounded parabolic actions, which provides a unified framework for the study of horoballs in weak cusped spaces.

Section~\ref{sec:UniformlyPerfect} focuses on uniform perfection of boundaries of proper $\delta$--hyperbolic spaces.
We give a proof of Proposition~\ref{prop:IntroUniformlyPerfect} and show that all weak cusped spaces have uniformly perfect boundary.

In Section~\ref{sec:HoroDistortion}, we define the closely related properties of constant horospherical distortion and constant horospherical dilation and prove their equivalence.
Proposition~\ref{prop:SymSpaceCHD} and Remark~\ref{rem:CombinatorialCHD} establish the results claimed in assertions (\ref{item:CanonCooper}) and (\ref{item:SymmetricCusped}) of Theorem~\ref{thm:introcuspedspaces}.

In Section~\ref{sec:carnot}, we review some relevant results about negatively curved symmetric spaces and Heintze groups of Carnot type. We conclude by demonstrating that horoballs in these spaces have constant horospherical distortion.

In Section~\ref{sec:CuspExtension} we prove several results involving cusp extensions of maps between various weak cusped spaces.
Theorem~\ref{thm:introcuspedspaces}(\ref{item:CanonicalQI}) is a direct consequence of Theorem~\ref{thm:CuspExtension}, and Theorem~\ref{thm:introcuspedspaces}(\ref{item:RelQC}) is proved in Corollary~\ref{cor:Quasiconvexity}.
We also prove Theorem~\ref{thm:introconformalgaugeconverse} and Corollary~\ref{cor:introconformalgauge} in this section (see Corollary~\ref{cor:conformalgauge}).

Section~\ref{sec:SymmSpaces} contains the proof of Theorem~\ref{thm:IntroRecognizingLattices}, which focuses on the specific case of rank one symmetric spaces.
The proof is divided into two parts, which are established in
Theorems \ref{thm:LatticeRelHyp}~and~\ref{thm:RecognizingLattices}.

Finally, an appendix contains a detailed examination of conical limit points. We give self-contained proofs that three different definitions of conical limit point found in the literature are equivalent.
The results of this appendix are widely known among experts, but the authors could not find a concise proof of this particular equivalence in the literature.

\subsection{Acknowledgements}
The authors thank Daniel Groves, Matthew Haulmark and Gabriel Pallier for helpful discussions and feedback that have improved the exposition of this article.
The second author was partially supported by grants \#318815 and \#714338 from the Simons Foundation.

\section{Hyperbolic spaces and quasisymmetric maps}
\label{sec:Quasisymmetric}

This section reviews visual metrics on the boundary of a $\delta$--hyperbolic space and the correspondence between quasi-isometries of $\delta$--hyperbolic spaces and quasisymmetric maps of their boundaries.
Many of the notions discussed below make sense for arbitrary $\delta$--hyperbolic spaces.  However for the sake of simplicity, we focus on the case of proper $\delta$--hyperbolic spaces.

Throughout this paper, we use the convention that the words ``ball'' and ``neighborhood'' refer to \emph{closed} balls and neighborhoods.

\begin{defn}[$\delta$--hyperbolic]
Let $X$ be a metric space. For points $x,y,z \in X$ the \emph{Gromov product} is defined by the rule
\[
   (y|z)_x = \frac{1}{2} \bigl( d(x,y) + d(x,z) - d(y,z) \bigr).
\]
The \emph{equiradial points} of a geodesic triangle $\Delta$ with vertices $x,y$ and $z$ are the unique points $p_x \in [y,z]$, \ $p_y \in [x,z]$, and $p_z\in [x,y]$ satisfying the following:
\begin{align*}
   d(x,p_y)=d(x,p_z) &= (y|z)_x, \\
   d(y,p_z)=d(y,p_x) &= (z|x)_y, \\
   d(z,p_x)=d(z,p_y) &= (x|y)_z.
\end{align*}
A geodesic metric space $(X,d)$ is \emph{$\delta$--hyperbolic} if for each geodesic triangle with vertices $x,y,z$ in $X$ the following holds:
if $y' \in [x,y]$ and $z'\in [x,z]$ are points such that $d(x,y') = d(x,z') \le (y|z)_x$, then $d(y',z') \le \delta$.
Note that if $X$ is $\delta$--hyperbolic then, for each geodesic triangle, any side lies in the closed $\delta$--neighborhood
of the union of the other two sides.
\end{defn}

\begin{defn}[Gromov boundary]
If $X$ is a proper $\delta$--hyperbolic space, the \emph{boundary} (or \emph{Gromov boundary}) $\boundary X$ of $X$ is the set of all equivalence classes of geodesic rays in $X$, where two rays $c,c'$ are considered equivalent if the distance $d\bigl(c(t),c'(t)\bigr)$ remains bounded as $t \to \infty$.
The set $X \cup \boundary X$ has a natural topology that is compact and metrizable (see, for example, Bridson--Haefliger \cite[\S III.H.3]{BH99}).
\end{defn}

\begin{defn}[Extended Gromov product]
The Gromov product is extended to $X \cup \boundary X$ as follows.
Suppose $X$ is $\delta$--hyperbolic, and suppose $x \in X$ and $y,z \in X \cup \boundary X$.  We define
\[
   (y|z)_x = \inf\ \liminf_{i\to \infty} (y_i|z_i)_x,
\]
where the infimum is taken over all pairs of sequences $y_i \to y$ and $z_i \to z$ of points in $X$.
If $y,z \in X$ this definition agrees with the previous definition.
\end{defn}

\begin{thm}[\cite{CDP90,ABC91}]
Let $X$ be a proper $\delta$--hyperbolic space.
The following ``Gromov inequality'' holds for all $x \in X$ and all $u,v,w \in X \cup \boundary X$:
\[
   (u|v)_x \ge \min \bigl\{ (u|w)_x, (v|w)_x \bigr\} - \delta
\]
\end{thm}

The following result is well known (see \cite[Prop.~2.2.2]{CDP90}), but we discuss the details in order to keep track of the constant obtained.

\begin{prop}[Thin ideal triangles]
\label{prop:idealthin}
Let $X$ be a proper $\delta$--hyperbolic space.
Suppose $x,y,z$ are distinct points of $X \cup \boundary X$.  Then there exists a geodesic triangle with vertices $x,y,z$.  For any such triangle, each side lies in the closed $5\delta$--neighborhood of the union of the other two sides.
\end{prop}

\begin{proof}
In $X$, any side of a geodesic quadrilateral lies in the closed $2\delta$--neighborhood of the union of the other three sides \cite[Lem.~1.2.3]{BuyaloSchroeder07}.
It follows that any pair of geodesics with the same endpoints in $X \cup\boundary X$ are at Hausdorff distance at most $2\delta$.
By the Arzel\`{a}--Ascoli Theorem, there exists a triangle $\Delta'$ with vertices $x,y,z$ obtained as a limit of a sequence of triangles with vertices in $X$.  Any side of $\Delta'$ lies in the closed $\delta$--neighborhood of the union of the other two sides.
The result follows since an arbitrary triangle $\Delta$ with the same vertices as $\Delta'$ has sides at Hausdorff distance at most $2\delta$ from the sides of $\Delta'$.
\end{proof}

\begin{defn}
A function $f\colon X \to Y$ between metric spaces is \emph{coarsely Lipschitz} if for every $R\ge 0$ there exists $S\ge 0$ such that
\[
   d(x,x')<R \Longrightarrow d\bigl( f(x),f(x') \bigr) < S.
\]
Two maps $f,f'\colon X \to Y$ are \emph{close} if $d(f,f') = \sup \bigset{d\bigl( f(x),f'(x) \bigr)}{x \in X}<\infty$.
A coarsely Lipschitz map $f \colon X\to Y$ is a \emph{coarse equivalence} if there exists a coarsely Lipschitz map $g \colon Y\to X$ such that $g\of f$ and $f\of g$ are close to the identity maps on $X$ and $Y$ respectively.
\end{defn}

\begin{defn}
\label{def:quasiisometry}
A function $f\colon X \to Y$ between metric spaces is \emph{large-scale Lipschitz} with parameters $\lambda\ge 1$ and $\epsilon\ge 0$ if for all $x_1,x_2\in X$ we have
\[
   d \bigl( f(x_1),f(x_2) \bigr) \le
   \lambda \, d(x_1,x_2) + \epsilon.
\]
A large-scale Lipschitz map $f \colon X \to Y$ with parameters $\lambda$ and $\epsilon$ is a \emph{$(\lambda,\epsilon)$--quasi-isometry} if there exists a large-scale Lipschitz map $g\colon Y \to X$ with the same parameters such that $f\of g$ and $g\of f$ are $\epsilon$--close to identity functions.
A function is a \emph{quasi-isometry} if it is a $(\lambda,\epsilon)$--quasi-isometry for some constants $\lambda$ and $\epsilon$. 
\end{defn}

Many metric spaces considered in this article are length spaces.  In that context the following notions are equivalent (see, for example, \cite[\S 1.3]{Roe_coarse}).

\begin{lem}
\label{lem:LengthSpaceLipschitz}
Let $f\colon X \to Y$ be a function from a length space $X$ to a metric space $Y$. The following are equivalent:
\begin{enumerate}
    \item There exist $R,S >0$ such that $d(x,x')<R \Longrightarrow d\bigl( f(x),f(x') \bigr) < S$.
    \item $f$ is coarsely Lipschitz.
    \item $f$ is large-scale Lipschitz.
\end{enumerate}
If $X$ and $Y$ are length spaces then any coarse equivalence is a quasi-isometry.
\end{lem}


\begin{defn}[Quasisymmetry]
\label{def:Quasisymmetry}
A homeomorphism $f \colon Z \to Z'$ between metric spaces is an \emph{$\eta$--quasisymmetry} or \emph{quasisymmetry with modulus $\eta$} if $\eta \colon [0,\infty) \to [0,\infty)$ is a homeomorphism such that
\[
   d(x,a) \le t\,d(x,b) \Longrightarrow 
   d\bigl( f(x),f(a) \bigr) \le \eta(t)\,d\bigl( f(x),f(b) \bigr)
\]
for each $t\ge 0$ and all points $x,a,b \in Z$.  
A group action on $Z$ is \emph{uniformly quasisymmetric} if there exists a modulus $\eta$ such that the action is by $\eta$--quasisymmetric homeomorphisms of $Z$.

Two metrics $d$ and $d'$ on a space $Z$ are \emph{quasisymmetric} if the identity map $(Z,d) \to (Z,d')$ is a quasisymmetry.
For a metric space $(Z,d)$, the \emph{conformal gauge} $\mathcal{G}(Z,d)$ is the set of all metrics $d'$ on $Z$ quasisymmetric to $d$.
\end{defn}

For each proper $\delta$--hyperbolic space $X$, the boundary $\boundary X$ admits a natural family of metrics, known as visual metrics.

\begin{defn}[Visual metric]
A metric $d$ on $\boundary X$ is a \emph{visual metric} with basepoint $w \in X$ and parameter $a>1$ if there are positive constants $c_1$ and $c_2$ such that
\[
   c_1 a^{-(\xi|\xi')_w} \le d(\xi,\xi')
     \le c_2 a^{-(\xi|\xi')_w}
\]
for all $\xi,\xi' \in \boundary X$.
\end{defn}

\begin{rem}[Quasisymmetric maps of boundaries]
\label{rem:hypbdryquasisym}
The boundary $\boundary X$ of a proper $\delta$--hyperbolic space $X$ admits a visual metric, and all visual metrics on $\boundary X$ are quasisymmetric (see Buyalo--Schroeder \cite{BuyaloSchroeder07}). 
In other words, $\boundary X$ has a unique  conformal gauge that contains every visual metric.

A map between visual boundaries is a \emph{quasisymmetry} if it is quasisymmetric with respect to this conformal gauge; \emph{i.e.}, if it is quasisymmetric with respect to some---hence every---choice of visual metrics.
\end{rem}

The following results provide an equivalence between quasi-isometries between proper $\delta$--hyperbolic spaces and quasisymmetries of their boundaries.

\begin{thm}[\cite{BuyaloSchroeder07}, Thm.~5.2.17]
\label{thm:QItoQS}
Let $f\colon X \to Y$ be a quasi-isometry of proper $\delta$--hyperbolic spaces.  Then $f$ naturally induces an $\eta$--quasisymmetric homeomorphism $\boundary f\colon \boundary X \to \boundary Y$ with modulus $\eta$ depending only on the given quasi-isometry and hyperbolicity constants and the choice of visual metrics on $\boundary X$ and $\boundary Y$. Close quasi-isometries induce the same boundary map.
\end{thm}

The reverse implication of this equivalence requires some slightly technical hypotheses on both the boundary and the space itself.

\begin{defn}
\label{def:visual}
A proper $\delta$--hyperbolic space $X$ is \emph{visual} if there exists a base point $w \in X$ and a constant $C<\infty$ such that for every $x \in X$ there is a geodesic ray based at $w$ that intersects the closed ball $B(x,C)$.
We note that $X$ is visual whenever it admits a cocompact group action (\emph{cf.} \cite{Paulin96}).
\end{defn}

Recall that a topological space is \emph{perfect} if it has no isolated points.

\begin{defn}
\label{def:uniformlyperfect}
A metric space $Z$ is \emph{uniformly perfect} if there exists a constant $\mu \in (0,1)$ such that for every $x \in Z$ and every $r > 0$ such that $B_r(x) \neq Z$, the set
\[
   B_r(x) \setminus B_{\mu r}(x)
\]
is nonempty.
Every connected metric space is uniformly perfect.  Furthermore uniform perfection is a quasisymmetrically invariant property among metric spaces by Tukia--V\"{a}is\"{a}l\"{a} \cite[Lem.~3.9]{TukiaVaisala80}.
\end{defn}

Uniform perfection is discussed in more detail in Section~\ref{sec:UniformlyPerfect}.

\begin{thm}[\cite{BuyaloSchroeder07}, Cor.~7.2.3]
\label{thm:QStoQI}
Let $X$ and $Y$ be proper $\delta$--hyperbolic spaces.
Suppose $X$ and $Y$ are visual and the boundaries $\boundary X$ and $\boundary Y$ are uniformly perfect.
Then each $\eta$--quasisymmetry $f\colon \boundary X \to \boundary Y$ can be extended to a $(\lambda,\epsilon)$--quasi-isometry $F\colon X \to Y$.
Any two such extensions $F$ and $F'$ satisfy $d(F,F')<\epsilon$.

The constants $\lambda$ and $\epsilon$ depend only on the hyperbolicity constant $\delta$, the choice of visual metrics on $\boundary X$ and $\boundary Y$\!, and the quasisymmetric modulus $\eta$.
\end{thm}

\begin{defn}
\label{def:Quasiaction}
A \emph{$(\lambda,\epsilon)$--quasi-action} $\rho$ of a group $\Gamma$ on a metric space $X$ associates to each $\gamma \in \Gamma$ a $(\lambda,\epsilon)$--quasi-isometry $\rho(\gamma) \colon X \to X$ such that 
the supremum distance between the functions $\rho(\gamma_1) \of \rho(\gamma_2)$ and $\rho(\gamma_1 \gamma_2)$ is less than $\epsilon$ for all $\gamma_1,\gamma_2 \in \Gamma$.
\end{defn}

Quasi-isometric conjugacy of two quasi-actions, defined below, formalizes the intuitive idea of a function that is equivariant ``up to a bounded error.'' 
See Tukia \cite{Tukia94} or Eisenberg \cite{Eisenberg_thesis} for more background on quasi-actions.

\begin{defn}
\label{def:roughequivariance}
Let $\rho$ and $\rho'$ be quasi-actions of $\Gamma$ on $X$ and $X'$ respectively, and let $f\colon X \to X'$ be a quasi-isometry.  Then $\rho$ is \emph{quasi-isometrically conjugate} to $\rho'$ via $f$ if there is a constant $\nu$ such that $d \bigl( f \of \rho(\gamma), \rho'(\gamma) \of f \bigr) < \nu$ for all $\gamma \in \Gamma$. If $f$ quasi-isometrically conjugates two given quasi-actions, then $f$ is $\nu$--\emph{roughly equivariant}. A quasi-isometry is \emph{roughly equivariant} if it is $\nu$--roughly equivariant for some number $\nu$.
\end{defn}

\section{Boundaries and convergence groups}
\label{sec:ConvergenceGroup}

When a group $\Gamma$ acts continuously, properly, and isometrically on a proper $\delta$--hyperbolic space $X$, the induced action on the boundary has strong topological properties, formalized in the notion of a convergence group action.
This section reviews the basic theory of convergence group actions.  See Tukia \cite{Tukia94} for more details.


\begin{defn}[Convergence groups]
Let $M$ be a compact, metrizable space containing at least three points.
A \emph{convergence group action} of $\Gamma$ on $M$ is an action by homeomorphisms such that for any sequence $(\gamma_i)$ of distinct elements in $\Gamma$
there is a subsequence $(\gamma_{n_i})$ such that either there is an element $\gamma \in \Gamma$ such that the homeomorphisms converge $\gamma_{n_i} \to \gamma$ uniformly on $M$, or there exist points $\zeta,\xi \in M$ such that
\[
   \gamma_{n_i} \big| \bigl( M\setminus\{\zeta\}\bigr) \to \xi
\]
uniformly on compact sets.
In the second case, such a subsequence is a \emph{collapsing subsequence} with \emph{attractive point} $\xi$ and satisfies the equivalent condition that
\[
   \gamma^{-1}_{n_i} \big| \bigl( M\setminus\{\xi\}\bigr) \to \zeta
\]
uniformly on compact sets.
The \emph{limit set} of the action is the set of all attractive points of collapsing sequences.
A convergence group is \emph{discrete} if every sequence of distinct elements contains a collapsing subsequence.
\end{defn}


\begin{defn}[Conical limit]
\label{def:ConicalLimit}
Suppose $\Gamma$ has a convergence group action on $M$.  A point $\zeta \in M$ is a \emph{conical limit point} if there exists a sequence $(\gamma_i)$ in $\Gamma$ and a pair of distinct points $\xi_0\ne\xi_1 \in M$ such that
\[
   \gamma_i \big| \bigl( M - \{\zeta\}\bigr) \to \xi_0
   \qquad \text{and} \qquad
   \gamma_i(\zeta) \to \xi_1.
\]
\end{defn}

\begin{defn}[Loxodromic]
\label{def:Loxodromic}
Suppose $\Gamma$ has a convergence action on $M$.
An infinite order element $\gamma$ is \emph{loxodromic} if $\langle \gamma \rangle$ has exactly two limit points.
\end{defn}

\begin{defn}[Parabolic]
A subgroup $P \le \Gamma$ of a convergence group is \emph{parabolic} if it has a unique limit point, called a \emph{parabolic point}.
A parabolic point $\eta$ is \emph{bounded parabolic} if $\Stab(\eta)$ acts cocompactly on $(M \setminus \{\eta\})$.
If $\Gamma$ is a discrete convergence group, then the stabilizer of a parabolic point is a maximal parabolic subgroup.
\end{defn}

\begin{defn}[Geometrically finite]
A discrete convergence group action on $M$ is \emph{geometrically finite} if every point of $M$ is either a conical limit point or a bounded parabolic point.
\end{defn}

As mentioned above, convergence groups often arise as the boundary actions of isometric actions (or quasi-actions) on $\delta$--hyperbolic spaces by the following result of Tukia (\emph{cf.}\ \cite{Freden95}).


\begin{thm}[\cite{Tukia94}, Thm.~3E]
\label{thm:ConvergenceGroup}
Let $X$ be a proper $\delta$--hyperbolic space $X$ such that $\boundary X$ contains at least three points.
Every quasi-action of a group $\Gamma$ on $X$ induces a convergence group action on $\boundary X$.
A sequence $(\gamma_i)$ in $\Gamma$ has a collapsing subsequence if and only if $\bigl\{\gamma_i(x)\bigr\}$ is an unbounded set for some $x \in X$.

In particular, a quasi-action by a discrete group is proper if and only if the induced convergence group action is discrete.
Furthermore, an isometric action on $X$ induces a convergence group action on $X \cup \boundary X$.
\end{thm}

When $\Gamma$ acts isometrically on a proper $\delta$--hyperbolic space $X$, the dynamical properties defined above all have simple geometric interpretations in terms of $\delta$--hyperbolic geometry, as explained below.

The notion of conical limit has several equivalent forms throughout the literature. 
See Appendix~\ref{sec:Conical} for a discussion of various equivalent definitions for conical limit points in terms of convergence groups and quasi-actions on $\delta$--hyperbolic spaces.

An isometry $\gamma$ of a proper $\delta$--hyperbolic space $X$ is \emph{loxodromic} if the induced map on the boundary is loxodromic or equivalently if for each $x \in X$ the orbit map $\Z \to X$ given by $n\mapsto \gamma^n(x)$ is a quasigeodesic \cite[Prop.~8.21]{GH90}.

A group $P$ acting isometrically on a proper $\delta$--hyperbolic space is \emph{parabolic} if the induced action on the boundary is parabolic.
In order to describe parabolic groups in geometric terms, we first review the notions of horofunctions and horoballs in $\delta$--hyperbolic spaces.
The following rough notion of horofunction was introduced by Bowditch \cite{Bowditch99Boundaries,BowditchRelHyp}.

Suppose $X$ is $\delta$--hyperbolic.  Let $\Delta=\Delta(x,y,z)$ be a triangle with vertices in $X \cup \boundary X$. An \emph{$L$--approximate center} of $\Delta$ is a point $w\in X$ such that the closed ball $B(w,L)$ intersects all three sides of $\Delta$. By Proposition~\ref{prop:idealthin}, each side of $\Delta$ contains a $5\delta$--approximate center.

\begin{defn}[Horofunctions]
\label{def:Horofunction}
Suppose $\eta \in \boundary X$.
A function $h \colon X \to \R$ is a \emph{horofunction about $\eta$} if
either of the following equivalent conditions holds
\begin{enumerate}
    \item
    \label{item:HorofunctionTriangle}
    There exists a constant $D_0<\infty$ such that if $\Delta(x,y,\eta)$ is a triangle with $5\delta$--approximate center $w$, then
\[
   \bigabs{\bigl( h(x) + d(x,w)\bigr) - \bigl( h(y) + d(y,w) \bigr) } < D_0.
\]
   \item
   \label{item:HorofunctionRay}
   There exists a constant $D_1<\infty$ such that if $x,w \in X$ and $d\bigl( w, [x,\xi) \bigr) \le 5\delta$ for some ray $[x,\xi)$,
   then the difference $h(w)-h(x)$ is within $D_1$ of the distance $d(x,w)$.
\end{enumerate}
\end{defn}

\begin{defn}
Let $X$ be a proper hyperbolic space.
An \emph{\textup{(}open\textup{)} horoball} about a point $\eta\in \boundary X$ is a subset $B\subset X$ that has the form $h^{-1}(-\infty,0)$ for some horofunction $h$ about $\eta$.
\end{defn}

\begin{prop}[Parabolic actions]
\label{prop:InvariantHoro}
Let $P$ be a group acting isometrically on a proper $\delta$--hyperbolic space $X$ with an unbounded orbit. The following are equivalent.
\begin{enumerate}
    \item
    \label{item:UniqueLimit}
    $P$ has a unique limit point in $\boundary X$.
    \item
    \label{item:NoLoxodromic}
    $P$ contains no loxodromic element.
    \item
    \label{item:InvariantHoro}
    There exists a $P$--invariant horofunction.
\end{enumerate}
\end{prop}

For an analogous characterization in different terms, see \cite[\S 4]{Manning08} or \cite[\S 3.C]{CapraceCornulierMonodTessera15_Amenable}.

\begin{proof}
Conditions (\ref{item:UniqueLimit}) and~(\ref{item:NoLoxodromic}) are equivalent, for instance, by \cite[\S 2]{Tukia94}.
Let us see that (\ref{item:InvariantHoro}) implies (\ref{item:NoLoxodromic}).  Suppose $P$ contains a loxodromic element $g$. Then $\boundary X$ must contain at least two points.
Choose any $x \in X$, and let $\xi \in \boundary X$ denote the limit point of the quasi-geodesic ray $\bigset{g^n(x)}{n \geq 0}$. Let $\zeta \in \boundary X$ be any point distinct from $\xi$. The quasi-geodesic ray determined by $g^n(x)$ is within a bounded Hausdorff distance of a geodesic line $c$ from $\zeta$ to $\xi$
by stability of quasigeodesics and thin triangles.
Therefore the sequence $g^n(x)$ lies in a tubular neighborhood of the geodesic $(\zeta,\xi)$. It follows by Definition~\ref{def:Horofunction}(\ref{item:HorofunctionTriangle}) that $g$ does not leave invariant any horofunction based at either $\xi$ or $\zeta$.

To complete the proof, we show that (\ref{item:UniqueLimit}) and (\ref{item:NoLoxodromic}) together imply (\ref{item:InvariantHoro}).
Suppose $P$ contains no loxodromic and has unique limit point $\eta \in \boundary X$.
Let $x \in X$ and $p \in P$ be arbitrary.
Let $(p_1,p_2,\dots)$ be any sequence in $P$ such that $d\bigl( x,p_i(x) \bigr) \to \infty$.
By Theorem~\ref{thm:ConvergenceGroup}, any subsequence of $(p_i)$ has a further collapsing subsequence, whose attractive point must equal $\eta$ by (\ref{item:UniqueLimit}).
It follows that $p_i(x) \to \eta$. Thus we have $\bigl( p_i(x) \big| \eta \bigr)_x \to \infty$ by \cite[Lem.~4.6(2)]{ABC91}.
Thus there exists $p'\in P$ such that $d\bigl( x,p'(x) \bigr) > d\bigl( x,p(x) \bigr)$ and $\bigl( p'(x) \big| \eta \bigr)_x > \bigl( p(x) \big| \eta \bigr)_x + \delta$.
Since $p,p'$, and $pp'$ are not loxodromic, it follows from \cite[Lem.~9.2.3]{CDP90} and the Gromov inequality that
\[
   d \bigl( x,p(x) \bigr) < 2 \bigl( p(x) \big| p'(x) \bigr)_x + 6\delta
   \le 2 \bigl( p(x) \big| \eta \bigr)_x + 8\delta.
\]
By \cite[Lem.~8.1]{GH90}, there exists a horofunction $h$ about $\eta$.
Definition~\ref{def:Horofunction}(\ref{item:HorofunctionTriangle}) implies that
\[
   h(x) - h\bigl( p(x) \bigr) \le d\bigl( p(x),w \bigr) - d(x,w) + D_0
   = d\bigl( x,p(x) \bigr) - 2\,d(x,w) + D_0
\]
if $w$ is any $5\delta$--approximate center of any triangle $\Delta$ with vertices $x$, $p(x)$, and $\eta$ such that $w$ lies on a geodesic $\bigl[ x,p(x) \bigr]$.
Consider the point $w \in \bigl[ x,p(x) \bigr]$ with $d(x,w) = \bigl( p(x) \big| \eta \bigr)_x$.
By \cite[Lem.~4.6(3)]{ABC91}, there exists a sequence $z_i \in X$ such that $z_i \to \eta$ and $\bigl( p(x) \big| z_i \bigr)_x \to \bigl( p(x) \big| \eta \bigr)_x$.
Passing to a subsequence if necessary, the point $w$ lies within a distance at most $5\delta$ from rays $[x,\eta)$ and $\bigl[ p(x),\eta \bigr)$ obtained as limits of the geodesic segments $[x,z_i]$ and $\bigl[p(x),z_i\bigr]$.
It follows that $w$ is a $5\delta$--approximate center of a triangle $\Delta$ with vertices $x$, $p(x)$, and $\eta$. In particular, we see that
\[
   h(x) - h\bigl( p(x) \bigr)
    \le d\bigl( x,p(x) \bigr) - 2\bigl( p(x) \big| \eta \bigr)_x + D_0
    < 8\delta + D_0.
\]
It follows that the $P$--invariant function given by $\bar{h}(x) = \inf_{p\in P} h\bigl( p(x) \bigr)$ is within a finite distance of the horofunction $h$ and thus is also a horofunction.
\end{proof}

\section{Relatively hyperbolic groups}
\label{sec:RelHyp}

This section reviews the definition of relative hyperbolicity in terms of cusp-uniform actions on $\delta$--hyperbolic spaces and an equivalent characterization in terms of the convergence action on the boundary. We also review a combinatorial construction due to Groves--Manning, inspired by an earlier construction of Cannon--Cooper, of a space admitting a cusp-uniform action.

A \emph{group pair} $(\Gamma,\mathbb{P})$ consists of a group $\Gamma$ and a finite family $\mathbb{P}$ of subgroups.

\begin{defn}
\label{def:relhyp}
Let $\GaP$ be a group pair and $X$ a proper, visual $\delta$--hyperbolic space.
A \emph{cusp-uniform action} of $\GaP$ on $X$ consists of a properly discontinuous, isometric action of $\Gamma$ on $X$ with the following properties.
The family of subgroups $\mathbb{P}$ is a set of representatives of the conjugacy classes of maximal parabolic subgroups.
There exists a $\Gamma$--equivariant family of disjoint open horoballs with union $U$ such that the induced action of $\Gamma$ on $X - U$ is cocompact.
The space $X$ is a \emph{weak cusped space} for the pair $\GaP$, and $X-U$ is a corresponding \emph{truncated space}.
\end{defn}

The stronger notion of a ``cusped space'' for $\GaP$ is often useful for the study of quasi-isometries and metrics on the boundary of $X$, in the case that $\Gamma$ is finitely generated.  This notion is introduced below in Section~\ref{sec:HoroDistortion}.

We also remark that if the condition in Definition~\ref{def:relhyp} holds for a non-visual space $X$, then it also holds for a visual space $Y \subseteq X$ (see \cite{BowditchRelHyp}).
We only consider cusp-uniform actions on visual spaces in this paper.

The following result, due to Bowditch, shows that the notions of cusp uniform and geometrically finite actions coincide.

\begin{thm}[\cite{BowditchRelHyp}]
\label{thm:CuspUiffGeomFin}
Let $\GaP$ be a group pair such that $\Gamma$ acts properly discontinuously and isometrically on a proper, visual $\delta$--hyperbolic space $X$.
The action of $\GaP$ on $X$ is cusp uniform
if and only if the induced action of $\Gamma$ on $\partial X$ is geometrically finite and $\mathbb{P}$ is a set of representatives of the conjugacy classes of maximal parabolic subgroups of the action.
\end{thm}

\begin{defn}[Relatively hyperbolic]
A group pair $\GaP$ is \emph{relatively hyperbolic} if it admits an action on a $\delta$--hyperbolic space $X$ that
satisfies either of the equivalent conditions of the previous theorem.
\end{defn}

\begin{defn}[Bowditch boundary as a topological space]
\label{def:BoundaryTop}
Let $\GaP$ be relatively hyperbolic.
The \emph{Bowditch boundary} $\partial\GaP$ is the Gromov boundary $\boundary X$ of any associated weak cusped space $X$.

By a theorem of Bowditch \cite{BowditchRelHyp}, any two weak cusped spaces $X$ and $Y$ have boundaries that are $\Gamma$--equivariantly homeomorphic.  Thus the Bowditch boundary is well-defined as a topological space equipped with an action of $\Gamma$ by homeomorphisms. (See Gerasimov--Potyagailo \cite{GerasimovPotyagailo16_Similar} for the case when $\Gamma$ is not finitely generated.)
\end{defn}

\begin{rem}[Lack of well-defined conformal gauge]
We caution the reader that if $X$ and $Y$ are two weak cusped spaces $X$ and $Y$ for a relatively hyperbolic pair $\GaP$, then there need not exist a roughly equivariant quasi-isometry $X \to Y$ by a theorem of the first author \cite{Healy}.
We show in Proposition~\ref{prop:CuspUniformlyPerfect} below that any weak cusped space has a uniformly perfect boundary.
Thus it follows from Theorem~\ref{thm:QStoQI} that different weak cusped spaces may have visual metrics that are not equivariantly quasisymmetric.
Thus it does not make sense in general to define a conformal gauge on the boundary using an arbitrary weak cusped space $X$.
\end{rem}

If $(\Gamma,\mathbb{P})$ is relatively hyperbolic and $\Gamma$ is finitely generated, then each $P \in \mathbb{P}$ is finitely generated (see \cite[Thm.~1.1]{Osin06}).
In this case, a weak cusped space for $(\Gamma,\mathbb{P})$ is given by the following combinatorial construction of Groves--Manning \cite{GrovesManning08DehnFilling}. (See also \cite{CannonCooper92,Rebbechi01,BowditchRelHyp}.)

\begin{defn}
\label{def:CombiHoroball}
Let $\Upsilon$ be a graph with the simplicial metric.  The \emph{combinatorial horoball} based on $\Upsilon$ is the graph $\mathcal{H}(\Upsilon)$, also equipped with the simplicial metric,
with vertex set $\Upsilon^{(0)} \times \Z_{\ge 0}$ and with two types of edges:
\begin{itemize}
   \item Vertices $(v,n)$ and $(v,n+1)$ are joined by a \emph{vertical edge}.
   \item Vertices $(v,n)$ and $(w,n)$ are joined by a \emph{horizontal edge} if the distance $d(v,w)$ in $\Upsilon$ is at most $e^n$.
\end{itemize}
\end{defn}

A finite generating set $\mathcal{S}$ for $\Gamma$ is \emph{adapted to $\mathbb{P}$} if for each subgroup $P \in \mathbb{P}$ the set $\mathcal{S} \cap P$ generates $P$.
In this case, for each $P\in \mathbb{P}$ and each $\gamma \in \Gamma$, the vertices of the coset $\gamma P$ span a connected full subgraph of $\Cayley(\Gamma,\mathcal{S})$ that is a $\gamma$--translate of the Cayley graph $\Cayley(P,\mathcal{S}\cap P)$.

\begin{defn}[Cusped Cayley graph]
\label{def:CuspedCayley}
Let $(\Gamma,\mathbb{P})$ be a group pair, and let $\mathcal{S}$ be a finite adapted generating set for $\Gamma$.
The \emph{cusped Cayley graph} $Y(\Gamma,\mathbb{P})$ is obtained from the Cayley graph $\Cayley(\Gamma,\mathcal{S})$ by attaching, for each $P \in \mathbb{P}$ and each $\gamma\in \Gamma$, a copy of the combinatorial horoball $\mathcal{H}\bigl( \gamma \Cayley(P,\mathcal{S}\cap P) \bigr)$ by identifying the subgraph $\gamma \Cayley(P,\mathcal{S}\cap P)$ with the $n=0$ level of the horoball.
\end{defn}

Groves--Manning have shown that relative hyperbolicity may be characterized in terms of the hyperbolicity of the cusped Cayley graph, according to the following theorem.

\begin{thm}[\cite{GrovesManning08DehnFilling}]
\label{thm:grovesmanning}
Let $(\Gamma,\mathbb{P})$ be a group pair, and let $\mathcal{S}$ be a finite adapted generating set for $\Gamma$. Then $\GaP$ is relatively hyperbolic if and only if the cusped Cayley graph for $\GaP$ is hyperbolic.
\end{thm}

\section{Bounded parabolic horoballs}
\label{sec:BoundedParabolic}

In this section, we introduce the notion of a bounded parabolic action on a $\delta$--hyperbolic space, which describes the way peripheral subgroups act on their corresponding horoballs in the context of cusp uniform actions by relatively hyperbolic pairs. We also prove some lemmas about the geometry of horoballs which admit these actions, which will be used in later sections.

We begin by observing that visual $\delta$--hyperbolic spaces satisfy the following ``tautness'' property.

\begin{lem}
\label{lem:VisualIffTaut}
Let $X$ be a proper $\delta$--hyperbolic space such that $\abs{\boundary X}\ge 2$. Then the space $X$ is visual if and only if there exists $C<\infty$ such that for any point $x \in X \cup \boundary X$ and any point $y \in X$ there exists a geodesic $c$ emanating from $x$ such that $d( y, c )< C$. 
\end{lem}

\begin{proof}
The reverse implication is trivial, so we focus on the forward implication. The union $H$ of all geodesic lines in $X$ is nonempty by hypothesis.  Note that $H$ is quasiconvex in $X$, and every geodesic ray emanating from a point of $H$ lies uniformly close to $H$.  Suppose $w$ is a basepoint as in Definition~\ref{def:visual}, and let $D$ be the distance from $w$ to $H$.  Since every geodesic ray emanating from a point in the $D$--neighborhood of $H$ lies uniformly close to that neighborhood, it follows that each $y\in X$ is uniformly close to the $D$--neighborhood of $H$.
We have thus shown that every point $y \in X$ is uniformly close to a geodesic line $\ell_y$.
If $x$ is an endpoint of $\ell_y$, we are done.
Otherwise the result follows by applying Proposition~\ref{prop:idealthin} to a triangle with $\ell_y$ as a side and with $x$ as a vertex.
\end{proof}

In the presence of a cusp-uniform action, the horoballs centered at a parabolic point have the following structure.

\begin{defn}[Bounded parabolic horoballs]
\label{def:BoundedParabolic}
Suppose $Y$ is a proper $\delta$--hyperbolic space such that the Gromov boundary $\boundary Y$ is a single point $\eta$.
Suppose a group $P$ acts isometrically on $X$.  
Fix a basepoint $y_0 \in Y$, and choose a geodesic ray $c\colon [0,\infty) \to Y$ from $y_0$ to $\eta$.
The action of $P$ on $Y$ is \emph{bounded parabolic} if $Y$ lies in a finite metric neighborhood of the union of all translates $pc$ for $p \in P$.
We note that the bounded parabolic property does not depend on the choice of basepoint $y_0$ or the choice of ray $c$.
\end{defn}

\begin{rem}
\label{rem:inducedhoroball}
Let $X$ be a $\delta$--hyperbolic space admitting a cusp-uniform action by $(\Gamma,\mathbb{P})$. Let $B$ be any element of the family of equivariant disjoint horoballs as in Definition~\ref{def:relhyp}. 
If $P \in \mathbb{P}$ stabilizes $B$, then the action of $P$ on $B$ is bounded parabolic by Lemma~\ref{lem:VisualIffTaut} together with a result of Bowditch \cite[Lem.~6.3]{BowditchRelHyp}.
\end{rem}

We will use $\approx_C$ to indicate an approximate equality that holds with a possible additive error of at most $C$.

\begin{lem}
\label{lem:AsynchToSynch}
Let $X$ be a hyperbolic space which admits a bounded parabolic action by a group $P$; label the unique boundary point $\eta$. There exists a function $D = D(L)$ such that if $c$ is any geodesic ray tending to $\eta$ such that the point $c(t)$ is within a distance at most $L$ of the ray $pc$ for some $p \in P$, then $c(t)$ is within a distance of $D$ of the point $pc(t)$.
\end{lem}

\begin{proof}

Let $h : X \rightarrow \R$ be a $P$--invariant horofunction about $\eta$. We claim that there exists a function $D = D(L)$ such that \ if $x,w \in X$ and $d\bigl( w, [x,\xi) \bigr) \le L$ for some ray $[x,\xi)$, then the difference $h(w)-h(x)$ is within $D(L)$ of the distance $d(x,w)$. This should be compared to Definition~\ref{def:Horofunction}(\ref{item:HorofunctionRay}). To verify this claim, one should consider a path between points $x,w$ broken up into segments of length $5 \delta$ and applying the properties of a horofunction iteratively to each pair.

Let $pc(t')$ be a point within a distance $L$ of $c(t)$. Without loss of generality, we may assume $t' > t$. Consider the subpath of $c$ with initial point $c(t)$, which is also a ray tending to $\eta$. The above claim gives a value $D = D(L)$ $d\bigl( c(t),pc(t) \bigr) \simeq_D h\bigl( c(t) \bigr) - h\bigl( pc(t)\bigr) = 0$. Therefore $d\bigl( c(t),pc(t) \bigr) < D$.
\end{proof}

For an arbitrary bounded parabolic horoball, $D$ will represent a constant satisfying both Lemma~\ref{lem:approxDistance} and Lemma~\ref{lem:AsynchToSynch} for $L = 5 \delta$.

\begin{lem}
\label{lem:lowcenters}
Let $Y$ be a horoball with bounded parabolic group action by $P$ and let $D$ be the constant guaranteed by Lemma~\ref{lem:AsynchToSynch}.
Let $c$ be a geodesic ray based at $y_0$, and choose any $p,p' \in P$. 
Suppose $d\bigl( p(y_0),p'(y_0) \bigr) \le D$.
Then for any $t \ge D + 5\delta$ we have
$d \bigl( pc(t), p'c(t) \bigr) \le D$.
\end{lem}

\begin{proof}
Assume by way of contradiction that $d \bigl( pc(t), p'c(t) \bigr) > D$.
Then by Lemma~\ref{lem:AsynchToSynch}, the distance from the point $pc(t)$ to the ray $p'c$ is greater than $5\delta$, so that $pc(t)$ lies within a distance $5\delta$ of a point on a geodesic $\bigl[ p(y_0), p'(y_0) \bigr]$, which has length at most $D$.  Thus $t\le D+5\delta$, as desired.
\end{proof}

\begin{lem}
\label{lem:approxDistance}
Let $Y$ be a horoball with bounded parabolic group action by $P$, and let $D$ be the constant guaranteed by Lemma~\ref{lem:AsynchToSynch}. There exists a constant $F$ that does not depend on the choice of $y_0$ such that the following holds.
Let $c$ be a geodesic ray based at $y_0$, and choose any $p,p' \in P$. Let $t_0$ be the minimum value such that $d \bigl( pc(t_0),p'c(t_0) \bigr) \leq D$.  Then $d \bigl(p(y_0),p'(y_0) \bigr) \approx_F 2t_0$.
\end{lem}

\begin{proof}
Begin by labeling the lone ideal point $\eta$ and the basepoint of the ray $c$ by $y_0$.
If $t=t_0 - 1$, then each of the points $pc(t)$ and $p'c(t)$ are within a distance $5\delta$ of respective points $z$ and $z'$ on a geodesic $\bigl[ p(y_0),p'(y_0) \bigr]$.
Indeed by Proposition~\ref{prop:idealthin}, the triangle $pc \cup p'c \cup \bigl[ p(y_0),p'(y_0) \bigr]$ is $5\delta$--thin. However, by Lemma~\ref{lem:AsynchToSynch} the distance from the point $pc(t)$ to the ray $p'c$ is greater than $5\delta$. A similar conclusion holds for $p'c(t)$.
As illustrated in Figure~\ref{fig:triangle},
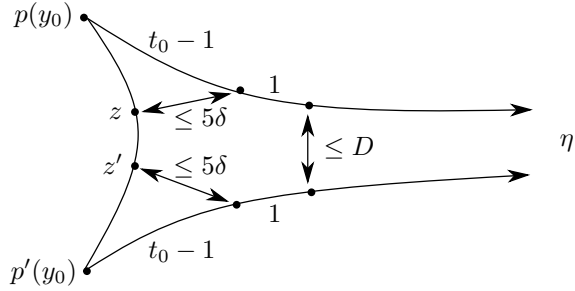
\begin{figure}
\labellist
\small
\pinlabel {$\leq 5 \delta$} at 52 53
\pinlabel {$\leq 5 \delta$} at 52 71
\pinlabel $p'(y_0)$ at -7 13
\pinlabel $p(y_0)$ at -7 110
\pinlabel $z$ at 20 73
\pinlabel $z'$ at 19 54
\pinlabel $\leq D$ at 108 60
\pinlabel $t_0-1$ at 45 100
\pinlabel $t_0-1$ at 45 21
\pinlabel $1$ at 80 35
\pinlabel $1$ at 80 84
\pinlabel $\eta$ at 190 62
\endlabellist
\begin{center}
\includegraphics{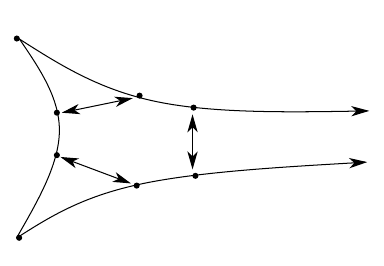}
\end{center}
\caption{A pair of asymptotic rays that differ by a parabolic isometry}
\label{fig:triangle}
\end{figure}
the triangle inequality give the following bounds:
\begin{align*}
   d(z,z') &\le D + 2 + 10\delta, \\
   d\bigl( p(y_0),z \bigr) &\ge t_0 - 1 - 5\delta, \\
   d\bigl( p'(y_0),z' \bigr) &\ge t_0 - 1 - 5\delta.
\end{align*}
The subsegments $\bigl[ p(y_0),z \bigr]$ and $\bigl[ z',p'(y_0) \bigr]$ are either disjoint or intersect in a segment of length at most $d(z,z')$.
Thus $d \bigl( p(y_0),p'(y_0) \bigr) \ge 2t_0 -4 -20\delta -D$.  We also note that, again as shown in Figure~\ref{fig:triangle},
we have
\begin{align*}
   d \bigl( p(y_0),p'(y_0) \bigr)
      &\le d \bigl( p(y_0), pc(t_0) \bigr) + \\
      &\quad\ d \bigl( pc(t_0), p'c(t_0) \bigr) +
      d_Y \bigl(p'c(t_0), p'(y_0) \bigr) \\
      &\le 2t_0 + D
\end{align*}
The conclusion now follows if we let $F = 4 + 20\delta + D$.
\end{proof}

\section{Uniformly Perfect Boundaries}
\label{sec:UniformlyPerfect}

We saw in Section~\ref{sec:Quasisymmetric} that in order to apply Theorem~\ref{thm:QStoQI} to pull metric information from the boundary of a hyperbolic space to its interior, we would need to know that the boundary is uniformly perfect.  In this section we prove Proposition~\ref{prop:UniformlyPerfect}, a result claimed by Bourdon--Kleiner that characterizes uniform perfection in terms of ideal triangles.
As an application we show in Proposition~\ref{prop:CuspUniformlyPerfect} that the boundary of any weak cusped space for a relatively hyperbolic group pair is uniformly perfect.

\begin{prop}[Meyer, Bourdon--Kleiner]
\label{prop:UniformlyPerfect}
Let $X$ be a proper, visual $\delta$--hyperbolic space such that $\boundary X$ contains at least two points.
The boundary $\boundary X$ is uniformly perfect if and only if there exists a constant $K<\infty$ such that each point $x\in X$ lies within distance at most $K$ from all three sides of some ideal triangle.
\end{prop}

The statement above in terms of ideal triangles is claimed without proof by Bourdon--Kleiner in \cite{BourdonKleiner15}.  
A similar result---without the properness hypothesis---is the main result of Meyer's dissertation \cite{Meyer_UniformlyPerfect}.
We recall the statement of Meyer's theorem, which is used in the proof below.

\begin{proof}
According to \cite{Meyer_UniformlyPerfect}, a visual $\delta$--hyperbolic space $X$ has uniformly perfect boundary if and only if $X$ is \emph{uniformly equilateral} in the following sense:
there exist two numbers $R_0,\mu>0$ such that for each $w \in X$ and $R>R_0$ the closed ball $B(w,R)$ contains three points $x_1,x_2,x_3$ such that
\[
   (x_i|x_j)_{x_k} \ge \mu R
   \qquad \text{whenever $\{i,j,k\}=\{1,2,3\}$.}
\tag{$\dagger$}
\label{eqn:Equilateral}
\]

We use Meyer's theorem to establish the forward implication first. Suppose $X$ is proper and uniformly equilateral, and choose any $w \in X$.
For each $R>R_0$ choose points $x_1,x_2,x_3$ satisfying \eqref{eqn:Equilateral} and choose a geodesic triangle $\Delta(x_1,x_2,x_3)$.
Let $q$ be a point within distance $\delta$ of each of the equiradial points of $\Delta(x_1,x_2,x_3)$, and for each $i$ let $y_i \in [x_i,x_{i+1}]$ be the unique point with $d(x_i,y_i) = (x_{i+1}|x_{i+2})_{x_i} - \mu R$.  (Indices are considered mod $3$.) Then by construction,
$\bigabs{d(q,y_i)-\mu R}<\delta$.
By visuality, there exist three geodesic rays $c_1,c_2,c_3$ based at $q$ with $d\bigl(y_i,c_i(\mu R) \bigr)<C+\delta$, where $C$ is given by Definition~\ref{def:visual}. By Lemma~\ref{lem:VisualIffTaut}, we see that the constant $C$ can be chosen independent of $q$.
On the other hand, for each $i\ne j$ we have $d(y_i,y_j) \ge 2\mu R - \delta$, which becomes arbitrarily large as $R \to \infty$.
Therefore we may choose $R>R_0$ large enough that
\[
   d \bigl( c_i(\mu R), c_j(\mu R) \bigr) > \delta
   \qquad \text{for all $i\ne j$.}
\]
By $\delta$--hyperbolicity, $\bigl( c_i(t) \big| c_j(t) \bigr)_q < \mu R$ for each $t$.
Thus there exists a geodesic triangle $\Delta_t$ with vertices $c_1(t), c_2(t), c_3(t)$ whose sides all intersect the closed ball $B(q,\mu R + \delta)$.
Passing to a subsequence if necessary, the triangles $\Delta_t$ limit to an ideal triangle $\Delta_\infty$ such that the original point $w$ is within a uniformly bounded distance of all three sides of $\Delta_\infty$.

Conversely, suppose there exists a universal constant $K$ such that for each point $w \in X$ there exists an ideal triangle $\Delta$ with vertices $\xi_1,\xi_2,\xi_3 \in \boundary X$ and points $q_1,q_2,q_3$ such that $q_i \in (\xi_{i-1},\xi_{i+1})$ and $d(w,q_i)<K$.
Let $\mu=1/2$ and $R_0 = 6K + 60 \delta$. Choose geodesic rays $c_i=[w,\xi_i)$.  Consider any $R> R_0$, and observe that, by definition, we have $R>K+5\delta$.
For any choice of $\{i,j,k\}=\{1,2,3\}$, consider the triangle $[w,q_k] \cup c_i \cup [q_k, \xi_i)$, where the third side is chosen to be a subset of $\Delta$.
By Proposition~\ref{prop:idealthin} this ideal triangle is $5\delta$--thin, and since $d(w,q_k)< K$, the point $x_i = c_i(R)$ is at most a distance $5 \delta$ from a point $y_{ij} \in [q_k,\xi_{i})$.
Thus the distance $d(y_{ij},q_k)$ is approximately equal to $R$ with an error at most $K+5\delta$.  Applying the above reasoning to all six ordered choices of distinct indices $i,j,k$, we see that each distance $d(y_{ij},y_{ji})$ is approximately equal to $2R$ with error at most $2(K+5\delta)$.
Since $d(x_i,y_{ij}) \leq 5\delta$, each distance $d(x_i,x_j)$ is approximately equal to $2R$ with error less than $2K + 20\delta$.
The definition of the Gromov product now gives the desired lower bound:
\[
   (x_i | x_j )_{x_k} > R-3K-30\delta > \frac{R}{2}=\mu R
   \quad \text{whenever} \quad
   \{i,j,k\}=\{1,2,3\}.
   \qedhere
\]
\end{proof}

In the case of a nonelementary word hyperbolic group $\Gamma$, the boundary $\boundary \Gamma$ is always uniformly perfect.
Indeed, Coornaert \cite{Coornaert93} shows that $\boundary \Gamma$ is Ahlfors regular, which easily implies uniform perfection (see, for example, Mackay--Tyson \cite[\S 1.4.3]{MackayTysonBook}).
The following result generalizes this result to proper hyperbolic spaces admitting a cocompact (not necessarily proper) group action.

\begin{prop}
\label{prop:HyperbolicUniformlyPerfect}
Let $X$ be a proper $\delta$--hyperbolic space admitting a nonelementary, cocompact, isometric group action. Then $\boundary X$ is uniformly perfect.
\end{prop}

\begin{proof}
Whenever $G$ acts cocompactly on $X$, its limit set must equal the entire boundary $\boundary X$.
If the action is nonelementary, then $\boundary X$ contains at least three points.
We verify the ideal triangle criterion in Proposition~\ref{prop:UniformlyPerfect}. By Proposition~\ref{prop:idealthin}, the space $X$ contains at least one ideal triangle $\Delta$, and there exists a point $p$ such that the closed ball $B(p,5\delta)$ intersects all three sides of $\Delta$.
By cocompactness, for any point $x \in X$ there exists $g \in G$ such that $x$ is within a uniformly bounded distance of $g(p)$ and hence of all three sides of the triangle $g(\Delta)$.
\end{proof}

\begin{prop}
\label{prop:CuspUniformlyPerfect}
Suppose $X$ is a non-elementary $\delta$--hyperbolic space that admits a cusp-uniform action such that
all peripheral subgroups are infinite and finitely generated. Then the Gromov boundary $\boundary X$ equipped with any visual metric is uniformly perfect.
\end{prop}

Since connected spaces are always uniformly perfect, Proposition~\ref{prop:CuspUniformlyPerfect} is primarily of interest for relatively hyperbolic pairs $(\Gamma,\mathbb{P})$ with disconnected boundary, which happens precisely when $\Gamma$ is not one-ended relative to $\mathbb{P}$ (see \cite[Prop.~10.1]{BowditchRelHyp}).

\begin{proof}
As in the cocompact case, we verify the criterion from Proposition~\ref{prop:UniformlyPerfect}. We will first assume that $X$ is visual, and deal with the non-visual case afterwards. Fix a cusp-uniform action by some group pair $\GaP$.

Since $\Gamma$ acts cocompactly on $X\setminus U$, each point of $X\setminus U$ lies within a uniformly bounded distance of $U$.  Furthermore, by visuality each point of $U$ lies within a uniformly bounded distance of a geodesic line tending to a parabolic point.  Thus, in order to establish the proposition, it suffices to show the following: there exists $D<\infty$ such that for each line $c$ tending to a parabolic point of $\boundary X$ and each point $w \in c \cap U$, there exists an ideal triangle having $c$ as a side such that $w$ is within a distance $D$ of all three sides of the triangle.

Fix a peripheral subgroup $P \in \mathbb{P}$ stabilizing the parabolic point $\xi \in \boundary X$ and choose a corresponding bounded parabolic horoball $B$ centered at $\xi$ (see Remark~\ref{rem:inducedhoroball}).  We claim that there exists $N<\infty$ such that for each $w \in B$ there exists $p \in P$ such that
\[
   D  + 2C< d \bigl( w, p(w) \bigr) < N,
\]
where $C$ is the constant given by Definition~\ref{def:visual} and $D$ is the constant from Lemma~\ref{lem:AsynchToSynch}. Let $w \in c$ be a point on a geodesic ray $c$ such that $P(c)=B$.

Now we claim that if $S$ is a finite, symmetric generating set for $P$ which includes the identity, there exists a value $R$ independent of $w$ such that for each $s \in S$ we have $d\bigl( w,s(w) \bigr) < R$. 
Indeed, by Bridson--Haefliger \cite[Lem.~III.H.3.3]{BH99}, there exists a value $T = T(x,y, \xi)$ such that for any asymptotic geodesic rays $c_x,c_y$ based at $x,y$ tending to $\xi$, it is true that $d\bigl( c_x(t),c_y \bigr) \leq T$ for all $t \geq 0$. We can apply Lemma~\ref{lem:AsynchToSynch} to conclude the existence of a (possibly greater) constant $T'$ such that $d\bigl( c_x(t),c_y(t) \bigr) \leq T'$. If $w = c(t)$ for some geodesic $c$, then $p(w) = pc(t)$ because $p$ acts by isometries. Thus $d\bigl( w,s(w) \bigr) \leq T'_s$ for some number $T'_s$ dependent on $s$. \
The claim follows by letting $R = \min \set{T'_q}{q \in S}$.

Since $P(w)$ is unbounded, there must exist $q \in P$ with $d\bigl( w,q(w) \bigr) > D + 2C$. Express $q$ as a product of generators $q = s_1 \cdots s_k$ where $s_i \in S$, and consider the sequence of points $w=w_0,w_1,\dots,w_k=q(w)$ defined by $w_i = s_1 \cdots s_i (w)$.
There must exist $j$ such that $d(w, w_j) > D + 2C$ but $d(w, w_{j-1}) \leq D + 2C$. 
Since $d(w_{j-1},w_{j}) = d\bigl( w,s_j(w) \bigr) < R$, we conclude that  $D + 2C < d(w, w_j) \leq D + 2C +R$. Choose $N$ such that $N> D + 2C +R$. Setting $p = s_1 \cdots s_i$ guarantees the claim.

For any point $w$ as in the claim above, the point $p(w)$ lies on a bi-infinite geodesic $c' = pc$ from $\xi$ to a boundary point $\eta' = p(\eta)$. We will see that $\eta' \neq \eta$. 
Recall that any pair of geodesic lines in a $\delta$--hyperbolic space with the same endpoints at infinity must be at a Hausdorff distance at most $2\delta$ from each other. Thus in order to conclude that $\eta\ne\eta'$ it suffices to show that $w$ is at a distance greater than $2 \delta$ from the closest point on $c'$.  
Assume for sake of contradiction that $\eta = \eta'$ and $w$ is within a distance $2\delta$ of some point $y$ on the geodesic $c'$. We observe by Proposition~\ref{prop:idealthin} that the bigon formed by $c$ and $c'$ is $2 \delta$--thin. Thus by Lemma~\ref{lem:AsynchToSynch}, we see that $d\bigl(w,p(w)\bigr)< D$, contradicting our assumption. Therefore $w \in c$ is not within $2 \delta$ of the side $c'$ of the ideal bigon, violating the thin triangles condition.

Now we may form the ideal triangle $\Delta$ with endpoints $\eta, \eta', \xi$. By the above argument, we observe that $w \notin \mathcal{N}_\delta (c')$. Thus we may conclude $w$ is within distance $5 \delta$ of the side $(\eta, \eta')$ of $\Delta$. Finally, $d\bigl( w, c' \bigr) \leq d \bigl( w, p(w) \bigr) \leq N$ by design, where $N$ and $C$ only depend on the orbit of $\xi$. Therefore $w$ is within a distance $N$ of all three sides of the triangle $\Delta$, as $N$ was assumed to be larger than $D$ and therefore $5 \delta$.

We conclude that every point of the horoball $B$ lies uniformly close to a center of an ideal triangle.  Since the horoballs of $U$ lie in finitely many orbits, the proposition is a consequence of Proposition~\ref{prop:UniformlyPerfect}.
\end{proof}

\section{Horospherical distortion and dilation}
\label{sec:HoroDistortion}

This section is a detailed study of constant horospherical distortion, a property introduced by Bowditch in his proof of the local connectedness of certain Bowditch boundaries \cite{Bowditch99Boundaries}, and constant horospherical dilation, a condition that is implicit in the definition of a combinatorial horoball.

In this section, we give precise definitions of these two properties in the setting of horoballs admitting a bounded parabolic group action.  The main result of the section is the equivalence of constant horospherical distortion and constant horospherical dilation.

While the main results of this paper involve actions of discrete groups, we formulate the definitions here in the more general setting of locally compact groups (using the framework promoted by Cornulier--de la Harpe in \cite{CornulierHarpe_book}) since the generalization involves little extra effort.  
We prefer to work in this more general setting because it simplifies the discussion of Heintze groups of Carnot type in Section~\ref{sec:carnot} below, although this more general framework is not strictly necessary for applications to negatively curved symmetric spaces.

The following definition is based on \cite{CornulierHarpe_book}, although the definition here is slightly simpler since we consider only second countable groups.

\begin{defn}
\label{defn:geoadapt}
Let $G$ be a second countable, locally compact group.
A metric $d$ on $G$ is \emph{adapted} if it is left invariant, proper, and compatible with the topology of $G$.
An adapted metric is \emph{geodesically adapted} if there exist constants $a >0, b\ge 0$, and $c> 0$ such that for each pair of points $x,x' \in G$ there exists a natural number $n\le a\,d(x,x') + b$ and a sequence $x=x_0,x_1, \ldots, x_n = x'$ such that $d(x_{i-1},x_{i})< c$ for $i=1,\dots,n$.
\end{defn}

\begin{rem}[Existence of metrics]
Suppose $G$ is a second countable, locally compact group.
If $G$ is compactly generated, then $G$ admits a geodesically adapted metric \cite[Prop.~4.B.9]{CornulierHarpe_book}.

If $G$ is a discrete group with finite generating set, the corresponding word metric is geodesically adapted. 

Any connected Lie group $G$ admits a left-invariant length metric compatible with the topology.
In fact, any such metric $d$ is geodesically adapted.
Indeed, since $G$ is locally compact, there exists $\epsilon>0$ such that the closed ball of radius $\epsilon$ about the identity---thus about any point---is compact, so that $d$ is complete.  Thus $d$ is a proper geodesic metric by the Hopf--Rinow Theorem \cite{BH99}.
\end{rem}

\begin{rem}[Uniqueness of metrics]
Let $G$ be a second countable, compactly generated, locally compact group.  Then all geodesically adapted metrics are quasi-isometric by \cite{CornulierHarpe_book}.
(The discrete case is the Fundamental Theorem of Geometric Group Theory, \emph{i.e.}, the classical Schwarz--Milnor lemma \cite{Milnor68}.)
\end{rem}

In the following lemma, we show that two groups with proper, continuous, bounded parabolic actions on the same horoball must be quasi-isometric.  Thus the large-scale geometry of the group is an intrinsic property of the horoball.
We note that this conclusion does not follow directly from the Schwarz--Milnor lemma, since the actions are not cobounded and the parabolic orbits in $Y$ are exponentially distorted.

\begin{lem}
\label{lem:SchwarzMilnorHoroball}
Let $Y$ be a proper $\delta$--hyperbolic space, and let $P$ and $P'$ be two second countable, compactly generated, locally compact groups, each with a continuous, proper, isometric, bounded parabolic action on $X$.
If $d_P$ and $d_{P'}$ are geodesically adapted metrics on $P$ and $P'$ respectively, then $(P,d_P)$ and $(P',d_{P'})$ are quasi-isometric.
\end{lem}

\begin{proof}
The proof of Proposition~\ref{prop:InvariantHoro} shows that $Y$ admits horofunctions $h$ and $h'$ that are invariant with respect to $P$ and $P'$ respectively such that the sup-distance between $h$ and $h'$ is finite.
It follows from the definition of bounded parabolic that for any basepoint $y_0\in Y$, the orbits $Py_0$ and $P'y_0$ are at a finite Hausdorff distance from each other.
Each of $P$ and $P'$ is coarsely equivalent to its orbit with the subspace metric induced from $Y$.  Since the orbits are quasi-isometric subspaces, it follows that $P$ and $P'$ are coarsely equivalent. Any coarse equivalence between geodesically adapted metrics is a quasi-isometry by \cite[Prop.~3.B.9]{CornulierHarpe_book}
\end{proof}

In the following two definitions, the space $Y$ is a proper $\delta$--hyperbolic space that admits a continuous, proper, isometric, bounded parabolic action by a second countable, compactly generated, locally compact group $P$. The reader who prefers discrete groups can simply consider the case that $P$ is finitely generated, and the action is proper, isometric, and bounded parabolic.
Let $y_0\in Y$ be a fixed basepoint, and $d_P$ be a geodesically adapted metric on $P$.

\begin{defn}[Constant horospherical distortion]
\label{def:Distortion}
We say $Y$ has \emph{constant horospherical distortion} if there exists a bilipschitz homeomorphism $g \colon [0,\infty) \to [0,\infty)$ and a constant $B$ such that the following holds for all distinct $p,p' \in P$:
\[
   \frac{1}{B} \, \exp{ g( d_Y (p(y_0),p'(y_0)) ) }
   \le d_{\mathcal{T}} (p,p')
   \le B \, \exp{ g ( d_Y (p(y_0),p'(y_0)) ) }
\]
\end{defn}

\begin{defn}[Constant horospherical dilation]
\label{def:Dilation}
Fix a ray $c$ in $Y$ based at $y_0$.
Let $D=D(5 \delta)$ be the constant given by Lemma~\ref{lem:AsynchToSynch}.
We say $Y$ has \emph{constant horospherical dilation} if there exist a bilipschitz homeomorphism $f \colon [0,\infty) \to [0,\infty)$ and a positive $A'$ such that for each $A$ there exists $D'=D'(A)$ so that the following holds.
Suppose $p,p'$ are distinct elements of $P$ and $t\ge 0$.
\begin{enumerate}
\renewcommand{\theenumi}{\alph{enumi}}

      \item
   \label{item:HoroImpliesWord}
  If $d_Y \bigl( pc(ft), p'c(ft) \bigr) \le D$ then $d_{\mathcal{T}} (p,p') \le A' e^{t}$. 
   \item
   \label{item:WordImpliesHoro}
   If $d_{\mathcal{T}}(p,p') \le A e^{t}$ then $d_Y \bigl( pc(ft), p'c(ft) \bigr) \le D'$. 
\end{enumerate}
A function $f$ as above is a \emph{dilation function} for $Y$.
\end{defn}

By Lemma~\ref{lem:SchwarzMilnorHoroball}, the two preceding definitions do not depend on the choice of group acting bounded parabolically on $P$.
In other words, constant horospherical distortion and dilation are intrinsic properties of a horoball.

We now show that constant horospherical distortion and constant horospherical dilation are equivalent to each other.

\begin{prop}
\label{prop:DistortionDilation}
Let $Y$ be a $\delta$--hyperbolic space admitting a bounded parabolic action by the finitely generated group $P$.
The following are equivalent:
\begin{enumerate}
    \item 
    \label{item:CHDis}
    $Y$ has constant horospherical distortion. 
    \item
    \label{item:CHDil}
    $Y$ has constant horospherical dilation.
\end{enumerate}
\end{prop}

\begin{proof}
We begin by assuming (\ref{item:CHDil}).
Consider a geodesic ray $c$ based at a point $y_0 \in Y$.
Let $g\colon [0,\infty) \to [0,\infty)$ be the homeomorphism given by $g(t) = f^{-1}(t/2)$.
Then $g$ is $\lambda$--bilipschitz for some $\lambda>0$.
Consider any distinct elements $p,p' \in P$.
Let $t_0 = t_0(p,p')$ be the smallest value of $t$ such that $d_Y(pc(ft),p'c(ft)) \leq D$ as in Lemma~\ref{lem:AsynchToSynch}.
By Definition~\ref{def:Dilation}(\ref{item:HoroImpliesWord}), it follows that $d_{\mathcal{T}}(p,p') \le A' e^{t_0}$. By Lemma~\ref{lem:approxDistance}, we have $2f(t_0) \le d_Y(p(y_0),p'(y_0)) + F$.
Applying the definition of $g$, we see that
\begin{equation}
\tag{$*$}
\label{eqn:UpperDistortion}
\begin{split}
    d_{\mct} (p,p') &\leq A' e^{t_0} \\
    &= A' \exp g\bigl( 2f(t_0) \bigr) \\
   & \leq A' \exp g \bigl( d_Y(p(y_0),p'(y_0)) + F \bigr) \\
  & \le B_1 \exp g\bigl( d_Y(p(y_0),p'(y_0)) \bigr),
\end{split}
\end{equation}
where $B_1 = A' e^{\lambda F}$, establishing the upper bound of Definition~\ref{def:Distortion}.

To establish the lower bound of Definition~\ref{def:Distortion}, let $D' = D'(1)$ be the constant given by Definition~\ref{def:Dilation}(\ref{item:WordImpliesHoro}).  Assume without loss of generality that $D' > D$.
Consider any distinct $p,p'\in P$, and let $t_0=t_0(p,p')$ be as defined above.
We consider two cases depending on whether $f(t_0) < \frac{1}{2}(D' + 1  + F)$ or not.
If $f(t_0) < \frac{1}{2}(D' + 1  + F)$ then $d_Y(p(y_0),p'(y_0)) \le 2f(t_0) + 2F \leq D' + 1 +F$.
Setting $B_2 = \exp \bigl( - g(D' + 1 +2F) \bigr)$ gives
\[
  B_2 \exp g\bigl( d_Y(p(y_0),p'(y_0))\bigr) \leq  B_2 \exp{ g(D' + 1 +2F)} = 1
   \le d_{\mathcal{T}}(p,p')
\]
since $p,p'$ are distinct.
On the other hand, if $f(t_0) \ge \frac{1}{2}(D' + 1  + F)$, then we may choose $t_1$ so that
$f(t_1) = f(t_0) - \frac{1}{2}(D' +1 + F)$, where $F$ is as in Lemma~\ref{lem:approxDistance}.  Applying Lemma~\ref{lem:approxDistance} with $c(ft_1)$ playing the role of $y_0$, gives
\[
   d_Y \bigl( pc(ft_1), p'c(ft_1) \bigr) \ge 2\bigl(f(t_0)-f(t_1)\bigr) - F = D'+1 > D',
\]
so by Definition~\ref{def:Dilation}(\ref{item:WordImpliesHoro}), we have
\[
   d_{\mathcal{T}}(p,p')
    > e^{t_1} \ge e^{t_0} e^{-\frac{\mu}{2}(D' +1 + F)}
\]
where $\mu$ is the bilipschitz constant for $f$.
As in \eqref{eqn:UpperDistortion}, it follows that 
\[
d_{\mathcal{T}}(p,p')
      > B_3 \exp g \bigl(d_Y ( p(y_0),p'(y_0) ) \bigr)
      \] 
for a constant $B_3$ not depending on the choice of $p$ and $p'$, completing the proof of~(\ref{item:CHDis}).

We now assume (\ref{item:CHDis}).  Let $B$ and $g$ be the constant and bilipschitz homeomorphism given by Definition~\ref{def:Distortion}.
Consider the bilipschitz map $f$ given by $f(t) = \frac{1}{2} g^{-1}(t)$. Choose distinct $p,p' \in P$.
First we will prove that the conclusion of Definition~\ref{def:Dilation}(\ref{item:HoroImpliesWord}) is satisfied for $t = t_0$, which is defined to be the minimal value for which, $d_Y\bigl( pc(ft),p'c(ft) \bigr) \leq D$.
Let $\lambda$ be the bilipschitz constant of $g$. 
By a calculation similar to \eqref{eqn:UpperDistortion},  the definition of constant horospherical distortion together with Lemma~\ref{lem:approxDistance} implies that 
\begin{align*}
   d_{\mathcal{T}} (p,p')
   &\le B \, \exp g \bigl( d_Y (p(y_0),p'(y_0)) \bigr)  \\
   &\le B \exp g \bigl( 2 f(t_0) +F \bigr) \\
   &\le B e^{ g (2 f(t_0)  )} e^{\lambda F} \\
   &= A' e^{t_0},
\end{align*}
where $A' = B e^{\lambda F}$.
If $t > t_0$ then $d_{\mathcal{T}}(p,p') \le A'e^{t_0} < A' e^t$.
On the other hand, if $t < t_0$, the hypothesis of Definition~\ref{def:Dilation}(\ref{item:HoroImpliesWord}) is never satisfied by the definition of $t_0$, so there is nothing to prove in that case.

To establish Definition~\ref{def:Dilation}(\ref{item:WordImpliesHoro}), let $A>0$ be any constant. 
Choose any $p,p' \in P$, and let $t_0$ be as above.
If $t \geq t_0 + D + 5\delta$, we apply Lemma~\ref{lem:lowcenters} with $pc(ft_0)$ playing the role of $y_0$ to conclude that $d_Y\bigl( pc(ft),p'c(ft) \bigr) \leq D$.
We now assume that  $t < t_0 + D + 5\delta$. Assume for some $p,p'$ we have $d_{\mathcal{T}} (p,p') \leq A e^t$.
Applying the triangle inequality 
to the sequence of points $pc(ft), pc(ft_0), p'c(ft_0), p'c(ft)$, we observe that $d_Y \bigl( pc(ft), p'c(ft) \bigr) \leq 2 \abs{t_0 -t} + D$. Therefore it suffices to show that $\abs{t_0 - t}$ has an upper bound independent of $p$ and $p'$.
If $t \geq t_0$, then $\abs{t_0-t} = t - t_0 < D + 5\delta$ by hypothesis.
On the other hand, if $t \leq t_0$, then because of constant horospherical distortion and Lemma~\ref{lem:approxDistance}, we see that 
\begin{align*}
   A e^t \ge   d_{\mathcal{T}} (p,p')
   &\ge  \frac{1}{B} \, \exp g\bigl( d_Y (p(y_0),p'(y_0)) \bigr) \\
   &\ge \frac{1}{B} \, \exp g\bigl( 2f(t_0) -F\bigr) \\
      &\ge \frac{1}{B} \, \exp \bigl( g( 2f(t_0)) - \lambda F \bigr)  \\
   &\ge \frac{1}{B} e^{-\lambda F} \, e^{t_0}
\end{align*}
and so
\[ 
 {t_0 - t} \leq \ln (A B  e^{\lambda F}).
\]
Therefore we have established both (\ref{item:HoroImpliesWord}) and (\ref{item:WordImpliesHoro}) of Definition~\ref{def:Dilation}, completing the proof of~(\ref{item:CHDil}).
\end{proof}

\begin{rem}[Combinatorial horoballs]
\label{rem:CombinatorialCHD}
Consider a combinatorial horoball $H(\Gamma(P))$ based on a finitely generated group with word metric $\mathcal{T}$.  For vertices $\gamma_1, \gamma_2 \in \Gamma(P)$, if $d_\mathcal{T} (\gamma_1, \gamma_2) \leq A e^n$, then the points $\bigl(\gamma_1, n + \lceil \log(A) \rceil \bigr)$ and $\bigl( \gamma_2, n + \lceil \log(A) \rceil \bigr)$ are connected by an edge. Thus we see that Property~(\ref{item:WordImpliesHoro}) of Definition~\ref{def:Dilation} holds for $D' = 2 \lceil \log(A) \rceil+1$. 
To see that Property~(\ref{item:HoroImpliesWord}) of Definition~\ref{def:Dilation} holds, observe that the graph metric on $\Gamma(P) \times \{ n \}$ is coarsely equivalent to its subspace metric in $H(\Gamma)$.
If $(\gamma_1,n)$ and $(\gamma_2,n)$ are connected by an edgepath of length $K$ in the graph $\Gamma(P) \times \{n\}$, then the vertices $(\gamma_1,0)$ and $(\gamma_2,0)$ are connected by an edgepath of length at most $K e^n$ in $\Gamma(P)\times \{0\}$.

Similar reasoning also applies to the slightly different horoballs defined in \cite{CannonCooper92,Rebbechi01,BowditchRelHyp}.
As explained in \cite{GrovesManningSisto}, the horoballs considered in those papers are all quasi-isometric to the combinatorial horoballs of \cite{GrovesManning08DehnFilling}.
Thus constant horospherical distortion of the other horoball models could also be derived as a consequence of Corollary~\ref{prop:converse} below.
\end{rem}

\begin{defn}[Cusped spaces]
\label{defn:cuspedspace}
Let $\GaP$ be a relatively hyperbolic pair. If $\GaP$ acts cusp-uniformly on a visual space $X$ such that each element of the set of $\Gamma$--invariant horoballs has constant horospherical distortion, then $X$ is a \emph{cusped space}.
\end{defn}

Note that the condition of being a cusped space is strictly stronger than that of being a weak cusped space, as discussed in the following remarks.

\begin{rem}[On smooth manifolds with pinched negative curvature]
Let $M$ be a Riemannian manifold with pinched negative sectional curvature.
Heintze--Im Hof prove in \cite[Prop.~3.2]{HeintzeImHof77} that horospherical distortion in $M$ has an upper exponential bound and a lower exponential bound, but in general these exponentials have different bases.
One can view this result as a version of the Rauch Comparison Theorem centered at an ideal point.
This phenomenon is heavily dependent on small-scale geometry, which is governed mainly by the sectional curvature.
\end{rem}

The weak conclusion of the Ideal Rauch Comparison Theorem above is not strong enough to imply constant horospherical distortion in general, as shown by the following example. 
For a more extreme failure of constant horospherical dilation, see \cite{Healy}.

\begin{rem}[Manifolds without constant horospherical distortion]
\label{rem:multiplywarped}
Consider a multiply-warped product of the torus $S^1 \times S^1$ with a ray $0\le t <\infty$ in which the first $S^1$ factor shrinks at a rate $e^{-t}$ and the second $S^1$ factor shrinks at a rate $e^{-bt}$ for any $b>1$.
The universal cover of such a cusp is a horoball that admits a bounded parabolic action by $\Z^2$, yet does not have constant horospherical distortion.

Any complete, noncompact, orientable hyperbolic $3$--manifold of finite volume admits a metric for each $b>1 $ and $\epsilon>0$ that agrees with the given constant curvature metric on a compact core, has a multiply-warped metric as above on each torus cusp, and has all sectional curvatures in the interval $(-b^2-\epsilon,-1+\epsilon)$.
Indeed, on each torus cusp there exists a smooth metric that has constant curvature $\kappa=-1$ on $T^2 \times [0,t_0]$ for some $t_0>0$, agrees with the multiply-warped metric above on $T^2 \times [t_1,\infty)$, and has all sectional curvatures in the desired interval.
For instance, such a metric on $S^1 \times S^1 \times [0,\infty)$ may be defined by
\[
   ds^2 = [e^{-t}dx]^2 + \bigl[ f(t)\,dy \bigr]^2 + dt^2,
\]
where the warping function $f$ interpolates between $e^{-t}$ and $e^{-bt}$.
If $f$ is chosen as in \cite[Appendix]{dpps1}, the desired bounds on sectional curvature follow from the convexity results of \cite[\S 4.1]{Petersen_Riemannian} and curvature computations for multiply-warped product metrics as in \cite[\S 4.2]{Petersen_Riemannian}.

Note that if the parameters $\epsilon$ and $b$ are chosen sufficiently small, one obtains examples as above with all sectional curvatures pinched in an arbitrarily small interval about $-1$.
\end{rem}

It is well known that negatively curved symmetric spaces have pinched negative curvature.  In the next section, we will see that the horoballs of such a symmetric space have constant horospherical distortion.

\section{Heintze groups of Carnot type}
\label{sec:carnot}

Although this article is largely focused on the geometry of discrete groups, in this section we examine Heintze groups, which are the negatively curved solvable Lie groups.  We show that Heintze groups of Carnot type always have constant horospherical dilation.
Consequently all negatively curved symmetric spaces have constant horospherical distortion, a result that is implicit in work of Schwartz \cite{Schwartz95}.
For background on Heintze groups of Carnot type, we refer the reader to Heintze \cite{Heintze74_Homogeneous} or Pansu \cite{Pansu89}.

\begin{defn}[Heintze groups of Carnot type]
A connected, solvable Lie group $S$ of dimension at least two is a \emph{Heintze group of Carnot type} if $S$ splits as a semidirect product $S = N \rtimes_\alpha \R$ where $N$ is a nilpotent Lie group whose Lie algebra $\mathfrak{n}$ has lower central series $(\mathfrak{n}^i)_{i \ge 1}$ satisfying the following condition.
For some nonzero constant $\lambda$, the action of $t \in \R$ on $\mathfrak{n}$ is given by $e^{t\alpha}$, where $\alpha$ is a derivation with matrix $\text{diag}(\lambda,2\lambda,\dots,j\lambda,\dots)$ with respect to a linear decomposition $\mathfrak{n} = \bigoplus_{j\ge 1} V_j$ satisfying $\mathfrak{n}^i = \bigoplus_{j \geq i} V_j$.

This linear decomposition is the \emph{Carnot grading} of $\mathfrak{n}$.
We note that the action of $t$ on the linear subspace $V_j$ is multiplication by $e^{tj\lambda}$ for each $j\ge 1$, and the subspace $V_1$ generates $\mathfrak{n}$ as a Lie algebra. 
\end{defn}

A Heintze group of Carnot type is simply connected and admits a left-invariant metric of negative sectional curvature \cite{Heintze74_Homogeneous}.
Each Heintze group of Carnot type has natural horocyclic coordinates, whose geometric properties are described below.

\begin{defn}[Horocylic coordinates]
\label{defn:HorocyclicCoords}
A Heintze group of Carnot type $S = N \rtimes_\alpha \R$ is homeomorphic to $N \times \R$, so $N$ is simply connected and the exponential map $\exp\colon \mathfrak{n}\to N$ is a diffeomorphism (see \cite[Thm.~XII.2.1]{Hochschild_LieGroups}).
The automorphism $e^{t\alpha}$ of $\mathfrak{n}$ induces a unique automorphism of $N$, also denoted $e^{t\alpha}$, such that $e^{t\alpha}\of\exp = \exp\of e^{t\alpha}$.
Therefore the two maps denoted by $e^{t\alpha}$ have the same formula in exponential coordinates.
The group law on the set $N\times \R$ has the form
\[
  (x_1,t_1)(x_2,t_2) = \bigl(x_1\,e^{-t_1\alpha}\!(x_2), t_1 + t_2\bigr).
\]
In particular, left-translation by $-t \in \R$ maps $(x,t)$ to $(e^{t\alpha}x,0)$.
Let $g_0$ be a left-invariant Riemannian metric on $N$, and endow $S$ with the left-invariant metric $g = g_t \oplus dt^2$, where $g_t=(e^{t\alpha})^*g_0$.
If the constant $\lambda$ is sufficiently large, this metric has negative sectional curvature by a result of Heintze \cite{Heintze74_Homogeneous}.
(The precise value of $\lambda$ does not affect the isomorphism type of $S$, so we assume without loss of generality that $\lambda$ is sufficiently large.)
Furthermore, all geodesic lines $c_x(t) = (x,t)$ converge to a common point $\omega$ as $t \to -\infty$.  Indeed let $d$ denote the Riemannian length metric on $S$, and let $\rho_t$ denote the length metric on $N_t = N\times \{t\}$ with its induced Riemannian metric $g_t$.
Then $\rho_t$ satisfies
\[
  d\bigl( c_x(t),c_y(t) \bigr) \le \rho_t(x,y) = \rho_0 \bigl( e^{t\alpha}x, e^{t\alpha}y \bigr).
\]
It follows that the natural homomorphism $S \to \R$ with kernel $N$ is a Busemann function on $S$ centered at $\omega$ whose associated horospheres are the sets $N_t$.
Furthermore the action of $N$ on each horoball $N \times (-\infty,t)$ is bounded parabolic.
\end{defn}

\begin{prop}
\label{prop:SymSpaceCHD}
Let $G = N \rtimes_\alpha \R$ be a Heintze group of Carnot type endowed with a left-invariant Riemannian metric as in Definition~\ref{defn:HorocyclicCoords}, and let $\omega \in \boundary G$ be the unique boundary point fixed by $G$.
Then each horoball in $G$ centered at $\omega$ has constant horospherical dilation.
\end{prop}

\begin{proof}
Let $\mathfrak{n} = T_e N$ be the Lie algebra of $N$, and let $\mathfrak{n} = \bigoplus V_j$ be its Carnot grading.
Let $H$ be the sub-bundle of the tangent bundle $TN$ consisting of all left translates of $V_1$ under the action of $N$.
A piecewise smooth curve in $N$ is \emph{horizontal} if it is tangent to $N$.
Let $g_0$ be the given left-invariant Riemannian metric on $N$.
Since $V_1$ generates the Lie algebra $\mathfrak{n}$, we may consider the \emph{sub-Riemannian distance} $d_0$ on $N$ given by the infimum of $g_0$--lengths of horizontal curves joining points.
This rule defines a geodesically adapted metric in the sense of Definition~\ref{defn:geoadapt} (see Sections 3.2, 3.3, and 7.4 of \cite{ABB_Subriemannian}).
By definition, the semidirect product action of $t\in\R$ on $V_1$ is multiplication by $e^{t \lambda}$.  Hence the action of $t$ on $N$ dilates the distance $d_0$ by a factor of $e^{t\lambda}$.

Similarly let $d_t$ denote the analogous sub-Riemannian distance on the horosphere $N_t$.
Then for any $x,y \in N$ we have
\[
   d_t(x,y) = d_0(e^{t\alpha}x,e^{t\alpha}y)
   = e^{t\lambda} d_0(x,y).
\]
Since $N_0$ is a closed subgroup of $G$, the distance on $N_0$ obtained by restricting the Riemannian length metric on $G$ is an adapted metric, and thus is coarsely equivalent (but not quasi-isometric) to $d_0$.
All metrics under discussion are invariant under the left-action of $G$, and hence are invariant under left-translation by any $t \in \R$.
Hence $d_t$ is coarsely equivalent to the distance on $N_t$ obtained by restricting the Riemannian length in $G$, with coarse Lipschitz parameters that do not depend on the choice of $t$.
Therefore the horoball $B = N\times (-\infty,0)$ has constant horospherical dilation, which completes the proof since each horoball $N\times (-\infty,t)$ is isometric to $B$. 
\end{proof}

Every negatively curved symmetric space $X$ may be given the structure of a Heintze group of Carnot type such that its symmetric Riemannian metric has the form described in Definition~\ref{defn:HorocyclicCoords} (see, for example, \cite{Heintze74_Homogeneous}).
Since the isometry group of $X$ acts transitively on the sphere at infinity, one obtains the following corollary.

\begin{cor}
Each horoball in a negatively curved symmetric space has constant horospherical distortion. \qed
\end{cor}

\section{Cusp extensions}
\label{sec:CuspExtension}

In this section we will study how maps between relatively hyperbolic pairs induce corresponding maps with desirable properties between their cusped spaces under the right conditions, and prove Theorem~\ref{thm:CuspExtension}. Recall that maps $f$ and $g$ between metric spaces are \emph{close}, denoted $f \cong g$, if they are at a finite distance in the supremum metric. 

\begin{prop}
\label{prop:HoroballExtension}
Let $Y_1$ and $Y_2$ be $\delta$--hyperbolic spaces with bounded parabolic actions by second countable, compactly generated, locally compact groups $P_1$ and $P_2$ respectively.
Assume each $Y_i$ has constant horospherical distortion.
Choose a basepoint $y_i \in Y_i$ for each $i$.
Any coarsely Lipschitz map $\phi\colon P_1 \to P_2$ induces a coarsely Lipschitz map $\hat\phi\colon Y_1 \to Y_2$,
in the sense that 
\[
   d\bigl( \hat{\phi} (py_1), \phi(p)y_2   \bigr)
\]
is uniformly bounded for all $p \in P_1$.
Furthermore, $\hat\phi$ satisfies the functorial properties:
$\hat\phi\hat\psi \cong \widehat{\phi\psi}$
and $\widehat{\id} \cong \id$.
\end{prop}

\begin{proof}
Choose a geodesic ray $c_i$ in $Y_i$ with initial point $c_i(0) = y_i$.
It suffices to define $\hat\phi$ on the quasidense subset of $Y_1$ consisting of the union of rays $\bigcup_{p \in P_1} pc_1$.
Let $f_i\colon [0,\infty) \to [0,\infty)$ be a dilation function for the constant horospherical dilation of $Y_i$.
Define $\hat\phi$ by the rule
\[
   \hat\phi\bigl( pc_1(f_1 t) \bigr)
   =
   \phi(p)c_2(f_2 t).
\]
Following the same argument as in the proof of Schwartz \cite[Lem.~5.3]{Schwartz95}, the definition of constant horospherical dilation and the definition of coarsely Lipschitz provide constants $D,D'>0$
such that whenever $a$ and $b$ lie in the domain of $\hat\phi$, we have
\begin{equation}
\tag{$*$}
\label{eqn:CoarseLip}
   d ( a,b ) < D
   \quad \Longrightarrow \quad
   d \bigl( \hat\phi(a),\hat\phi(b) \bigr) < D',
\end{equation}
which implies that $\hat\phi$ is coarsely Lipschitz by Lemma~\ref{lem:LengthSpaceLipschitz}.

We have not shown that $\hat\phi$ is well-defined, since a point $a$ could lie in the intersection of two rays $pc \cap p'c$.  In this case, the definition of $\hat\phi$ requires making an arbitrary choice of ray $pc$ containing $a$.
The proof of \eqref{eqn:CoarseLip} shows that any two such choices lead to values $\hat\phi(a)$ and $\hat\phi'(a)$ that are separated by a distance less than $D'$.
Therefore the closeness equivalence class of $\hat\phi$ is uniquely determined and the stated functorial properties hold regardless of the arbitrary choices used in the various extensions.
\end{proof}

Applying Lemma~\ref{lem:LengthSpaceLipschitz}, we get the following conclusion regarding the cusp-extension of quasi-isometries.

\begin{cor}
\label{cor:HoroballExtensionQI}
Let $Y_1$ and $Y_2$ be $\delta$--hyperbolic with bounded parabolic actions by second countable, compactly generated, locally compact groups $P_1$ and $P_2$.  Suppose $Y_1$ and $Y_2$ each have constant horospherical distortion.
Each quasi-isometry $P_1\to P_2$ induces a quasi-isometry $Y_1 \to Y_2$.
An isomorphism $P_1\to P_2$ induces a roughly equivariant quasi-isometry $Y_1 \to Y_2$. \qed
\end{cor}

In the following definition, we consider finitely generated groups equipped with a word metric for a finite generating set.
By Osin \cite{Osin06}, if $(\Gamma,\mathbb{P})$ is relatively hyperbolic and $\Gamma$ is finitely generated, then so is each $P \in \mathbb{P}$.

\begin{defn}
\label{def:mapofpairs}
Let $\GaP$ and $\GaPp$ be relatively hyperbolic such that $\Gamma$ and $\Gamma'$ are finitely generated. A \emph{coarsely Lipschitz map of pairs} is a map $\phi \colon \Gamma \rightarrow \Gamma'$ that is coarsely Lipschitz with respect to the word metric
such that there exists a constant $R$ with the following property: for each $P \in \mathbb{P}$ and $\gamma \in \Gamma$, there exists $P' \in \mathbb{P}'$ and $\gamma' \in \Gamma'$ such that $\phi(\gamma P) \in \mathcal{N}_R(\gamma'P')$. 
(We note that in this setting all coarsely Lipschitz maps are large-scale Lipschitz \cite[Prop.~3.B.9]{CornulierHarpe_book}.)
A \emph{quasi-isometry of pairs} is a coarsely Lipschitz map of pairs that has a quasi-inverse which is also a coarsely Lipschitz map of pairs.
\end{defn}

The next theorem follows directly from Proposition~\ref{prop:HoroballExtension} and Corollary~\ref{cor:HoroballExtensionQI}.

\begin{thm}[Cusp extension]
\label{thm:CuspExtension}
Let $\GaP$ and $\GaPp$ be relatively hyperbolic pairs which admit cusp-uniform actions on respective cusped spaces $X$ and $X'$.
Every coarsely Lipschitz map of pairs $\phi \colon \GaP \to \GaPp$ induces a coarsely Lipschitz map $\hat{\phi}\colon X \to X'$.

Any quasi-isometry of pairs $\GaP \to \GaPp$ induces a quasi-isometry $X \to X'$ and therefore also a quasisymmetric homeomorphism $\partial X \to \partial X'$.

Furthermore, if $\GaP = \GaPp$ then any isomorphism $\GaP \to \GaPp$ induces a roughly equivariant quasi-isometry $X \to X'$ and an equivariant quasisymmetry.
\qed
\end{thm}

In the above set up, the map $\hat{\phi}$ is a \emph{cusp extension} of $\phi$. 

We get the following as a corollary of Theorem~\ref{thm:CuspExtension}, which is a generalization to the relatively hyperbolic setting of a well known result about hyperbolic groups.

\begin{cor}[Conformal Gauge]
\label{cor:conformalgauge}
Let $\GaP$ be a relatively hyperbolic pair with finitely generated peripheral subgroups, and let $X$ be  cusped space for $\GaP$. If $\mathcal{G}$ is the conformal gauge of $\partial X$, then whenever $Y$ is a cusped space for $\GaP$ such that $\partial Y$ is equipped with a visual metric, we have $\partial Y \in \mathcal{G}$.
\end{cor}

If $\GaP$ is relatively hyperbolic, its most natural family of subgroups from a geometric point of view are the relatively quasiconvex subgroups (see \cite{Hruska10RelQC}).
A relatively quasiconvex subgroup $H$ has a natural induced peripheral structure $\mathbb{Q}$ such that the pair $(H,\mathbb{Q})$ is itself relatively hyperbolic.

The following corollary interprets a theorem of Manning--Mart\'{i}nez in the context of this paper.
Manning--Mart\'{i}nez \cite{ManningMartinez10} prove this result for the cusped Cayley graphs of \cite{GrovesManning08DehnFilling}. The general case follows immediately from Remark~\ref{rem:CombinatorialCHD} and Theorem~\ref{thm:CuspExtension}.

\begin{cor}[Relative quasiconvexity]
\label{cor:Quasiconvexity}
Let $(\Gamma,\mathbb{P})$ be relatively hyperbolic, and let $(H,\mathbb{Q})$ be a relatively quasiconvex subgroup with its induced relatively hyperbolic structure.
Suppose $\Gamma$ and $H$ are each finitely generated, and
suppose $(H,\mathbb{Q})$ and $(\Gamma,\mathbb{P})$ have cusped spaces $Y$ and $X$ respectively.
Then the inclusion $H \to \Gamma$ induces a quasi-isometric embedding $Y \to X$ and a roughly $H$--equivariant quasisymmetric embedding $\boundary Y \to \boundary X$. \qed
\end{cor}

One may wonder if constant horospherical distortion is necessary, as well as sufficient, to be quasi-isometric to a cusped space. In fact, being quasi-isometric to a combinatorial horoball is strong enough to ensure a given bounded parabolic horoball has constant horospherical distortion.

\begin{prop}
\label{prop:converse}
Let $Y_1$ and $Y_2$ be horoballs each with a bounded parabolic action by the second countable, compactly generated, locally compact group $P$, and suppose there exists a roughly equivariant quasi-isometry $\hat\phi\colon Y_1 \to Y_2$.  If $Y_1$ has constant horospherical distortion, then so does $Y_2$.
\end{prop}

\begin{proof}
Recall that any combinatorial horoball for $P$ admits a bounded parabolic action and has constant horospherical distortion.
Therefore by Corollary~\ref{cor:HoroballExtensionQI}, we may assume without loss of generality that $Y_1$ itself is a combinatorial horoball and thus it satisfies Definition~\ref{def:Dilation} with dilation function equal to the identity.
This assumption---while not essential---simplifies the calculations throughout the rest of the proof.  According to Proposition~\ref{prop:DistortionDilation}, to establish constant horospherical distortion for $Y_2$, it suffices to show that $Y_2$ has constant horospherical dilation.

Suppose $\hat\phi$ is a $K$--roughly equivariant $(\lambda,\epsilon)$--quasi-isometry (as in Definition~\ref{def:roughequivariance}), and let $c_1$ be a geodesic ray in $Y_1$.
Since $\hat\phi \of c_1$ is a quasigeodesic, it lies within a finite Hausdorff distance of a geodesic ray $c_2$ emanating from the same initial point.
In particular, there exists a quasi-isometry $f \colon [0,\infty)\to [0,\infty)$ and a global constant $H$ such that
\[
   d \bigl( \hat\phi c_1 (t), c_2(ft) \bigr) < H.
\]
We may assume without loss of generality that $f$ is a bilipschitz homeomorphism, since each quasi-isometry of $[0,\infty)$ is close to a bilipschitz homeomorphism and the closeness depends only on the given quasi-isometry constants (see, for example, \cite{Sankaran06}).

By our choice of $f$ and $\hat\phi$, it follows that
\begin{align*}
   d \bigl( p c_2(ft), p'c_2(ft) \bigr)
   &\leq
   d \bigl( p \hat\phi c_1(t), p' \hat\phi c_1 (t) \bigr) +2H \\
   & \leq d \bigl( \hat\phi pc_1 (t), \hat\phi p'c_1 (t) \bigr) +2H +2K \\
 &\leq   \lambda\,d \bigl( pc_1(t), p'c_1(t) \bigr) + 2H + 2K + \epsilon.
\end{align*}
Since $Y_1$ satisfies Definition~\ref{def:Dilation}(\ref{item:WordImpliesHoro}) with dilation function equal to the identity, it follows that $Y_2$ satisfies the same property with dilation function $f$.

Applying a similar argument in reverse gives
\[
   d \bigl( pc_1(t),p'c_1(t) \bigr)
   \le \lambda\,d \bigl( p c_2(ft), p' c_2(ft) \bigr) + M,
\]
where $M$ is a constant depending only on $H,K,\lambda$, and $\epsilon$.
To show that $Y_2$ satisfies Definition~\ref{def:Dilation}(\ref{item:HoroImpliesWord}), we assume that
\[
   d \bigl( pc_2 (f (t)), p'c_2 (f ( t)) \bigr) \leq D,
\]
which immediately implies
\[
   d \bigl( pc_1(t),p'c_1(t) \bigr) \leq \lambda D + M.
\]
By Lemma~\ref{lem:approxDistance}, there is a constant $F$ such that the distance $d \bigl( pc_1(t),p'c_1(t) \bigr)$ is approximately equal to $2 (t_0-t)$ with error at most $F$, where $t_0$ is a parameter greater than or equal to $t$ satisfying
\[
   d \bigl( pc_1(t_0),p'c_1(t_0) \bigr) \leq D.
\]
Applying Definition~\ref{def:Dilation}(\ref{item:HoroImpliesWord}) to the space $Y_1$ gives a constant $A'_1$ such that $d_\mct (p,p') \leq A'_1 e^{t_0} = A'_1 e^{t_0-t} e^t$.
Since $t_0 - t$ is bounded in terms of $\lambda,D,M$, and $F$, we conclude that Definition~\ref{def:Dilation}(\ref{item:HoroImpliesWord}) also holds in $Y_2$.
\end{proof}

By Proposition~\ref{prop:converse}, we see that for a given relatively hyperbolic pair, a weak cusped space without constant horospherical distortion cannot be equivariantly quasi-isometric to a cusped space. 
We also get the following corollary as a direct application of Proposition~\ref{prop:converse}.

\begin{cor}
\label{cor:cuspedspacerigidity}
Let $\GaP$ be a relatively hyperbolic pair that acts cusp-uniformly on a cusped space $X_1$ and also on a weak cusped space $X_2$. If there is a roughly-equivariant quasi-isometry from $X_1$ to $X_2$, then $X_2$ is also a cusped space.
\end{cor}

When $\GaP$ is relatively hyperbolic and $\Gamma$ is finitely generated, Bowditch has shown that any two weak cusped spaces have equivariantly homeomorphic boundaries (see Definition~\ref{def:BoundaryTop}).

We are now ready for the proof of Theorem~\ref{thm:introconformalgaugeconverse}.

\begin{proof}[Proof of Theorem~\ref{thm:introconformalgaugeconverse}]
Recall that all visual metrics on each $\partial X_i$ are quasisymmetric. By assumption, $X_i$ are both visual spaces. Additionally, we may apply Proposition~\ref{prop:CuspUniformlyPerfect} to observe that both $\partial X_i$ are also uniformly perfect. Thus Theorem~\ref{thm:QStoQI} states that there is a quasi-isometry between $X_1$ and $X_2$. This quasi-isometry is also roughly equivariant by the assumption that the quasisymmetry is equivariant. Therefore we may apply Corollary~\ref{cor:cuspedspacerigidity} to observe that $X_2$ is also a cusped space.
\end{proof}

\section{Recognizing rank one lattices}
\label{sec:SymmSpaces}

The study of relatively hyperbolic groups grew out of generalizing the behavior of nonuniform lattices in negatively curved symmetric spaces; 
\emph{i.e.}, the classical hyperbolic spaces $\Hyp_{\mathbb{K}}^n$ where $\mathbb{K}$ is either the real, complex, quaternion, or octave numbers. In particular, lattices in rank one symmetric spaces are relatively hyperbolic with respect to a natural peripheral structure.

In \cite{CannonCooper92}, Cannon--Cooper show that lattices in real hyperbolic $3$--space (with their natural peripheral structure) can be recognized up to quasi-isometry as precisely those group pairs whose corresponding cusped Cayley graph is quasi-isometric to $\Hyp^3_{\R}$.
(Cannon--Cooper use a slightly different model for the cusped Cayley graph than defined here, but the two models are quasi-isometric by a result of \cite{GrovesManningSisto}.)

In this section we extend that result to all rank one symmetric spaces of noncompact type.

\begin{thm}
\label{thm:LatticeRelHyp}
Let $X$ be a negatively curved symmetric space.  Let $\Gamma$ be a lattice in $\Isom(X)$, and let $\mathbb{P}$ be a set of representatives of the conjugacy classes of maximal parabolic subgroups of $\Gamma$.\footnote{A subgroup $P\le\Gamma$ is \emph{maximal parabolic} if $P$ is the stabilizer of a parabolic point of the convergence action on $\boundary X$, as discussed in Section~\ref{sec:RelHyp}.
We caution the reader that this terminology conflicts with the meaning of ``parabolic'' in Lie theory.}
Then $\GaP$ is relatively hyperbolic, $\Gamma$ and each $P \in \mathbb{P}$ are finitely generated, and there exists a roughly equivariant quasi-isometry between the symmetric space $X$ and the cusped Cayley graph $Y(\Gamma,\mathbb{P})$.
\end{thm}

\begin{proof}
In \cite{GarlandRaghunathan70}, Garland--Raghunathan prove that any such lattice $\Gamma$ acts cusp-uniformly on the symmetric space $X$ with a connected truncated space $X \setminus U$. Thus $(\Gamma,\mathbb{P})$ is relatively hyperbolic, $\Gamma$ is finitely generated, and each $P \in \mathbb{P}$ is finitely generated. Furthermore, $X$ is a cusped space for $\GaP$, since $X$ is geodesically complete and all horoballs have constant horospherical dilation by Proposition~\ref{prop:SymSpaceCHD}.

Observe that there exists an equivariant quasi-isometry from the Cayley graph for $\Gamma$ with respect to any finite generating set to the truncated space $X-U$. This quasi-isometry induces a roughly equivariant quasi-isometry from the cusped space $Y(\Gamma,\mathbb{P})$ to $X$ by Theorem~\ref{thm:CuspExtension}.
\end{proof}

\begin{rem}[Finite extensions of lattices]
\label{rem:FiniteExtension}
We note that the conclusions of Theorem~\ref{thm:LatticeRelHyp} also hold more generally when $\Gamma$ is a finite extension of a lattice (in the sense that $\Gamma$ maps onto a lattice with a finite kernel), since a finite extension of a lattice also acts cusp-uniformly on $X$.
\end{rem}

We can use rigidity properties enjoyed by symmetric spaces to prove the following converse to the above statement. While the following statement may feel similar in flavor to the main results of Schwartz in \cite{Schwartz95}, we note to the reader that we make no assumption that the group is a priori quasi-isometric to a lattice, as is the hypothesis in \cite{Schwartz95}. Instead, this is part of our conclusion, which we obtain only from geometric information about a cusped space for $\Gamma$.

\begin{thm}
\label{thm:RecognizingLattices}
Let $\GaP$ be a group pair such that $\Gamma$ has a finite generating set $\mathcal{S}$ adapted to $\mathbb{P}$. If the cusped Cayley graph $Y\GaP$ is quasi-isometric to a rank one symmetric space $X$ then $\Gamma$ has a finite normal subgroup $F$ such that $\Gamma/F$ is isomorphic to a lattice in $\Isom(X)$ and $\mathbb{P}$ represents the conjugacy classes of preimages in $\Gamma$ of the maximal parabolic subgroups of the lattice.
\end{thm}

The uniform case was previously known, due to the combined efforts of many authors.
A key step in the uniform case is the following proposition, which has been used in the study of cobounded quasi-actions \cite{Tukia94,KleinerLeeb01}, but which also applies to certain actions that are not cobounded, as explained in the following proposition.
Although our main results on relative hyperbolicity use only the discrete case of this proposition, we include the nondiscrete case, which follows by essentially the same argument.

\begin{prop}
\label{prop:Quasiconjugate}
Let $X$ be a negatively curved symmetric space.
Suppose $G$ has a quasi-action on $X$ whose set of conical limit points in $\boundary X$ has positive measure.\footnote{Recall that in a smooth manifold, a set $A$ has \emph{measure zero} if for each chart $\phi\colon U \to \R^n$ the set $\phi(A\cap U)$ has Lebesgue measure zero in $\R^n$.  Thus any countable subset of a smooth manifold has measure zero.
The boundary sphere of a rank one symmetric space $X$ has a natural smooth structure as a $G$--homogeneous space $G/G_\xi$, where $G$ is the identity component of $\Isom(X)$ and $G_\xi$ is the stabilizer of a boundary point $\xi$.
See \cite{Lee_Smooth} for background on sets of measure zero and homogeneous spaces.}
Then the given quasi-action is quasi-isometrically conjugate to an isometric action on $X$.
\end{prop}

We note that in Section~\ref{sec:ConvergenceGroup} the notion of conical limit was defined only in the setting of proper, isometric actions.  This definition has a natural generalization to arbitrary quasi-actions, which is discussed in Appendix~\ref{sec:Conical}.

\begin{proof}
Note that $\boundary X$ is uniformly perfect, since it is connected.
By Theorems \ref{thm:QItoQS} and~\ref{thm:QStoQI}, the proposition is equivalent to the statement that any uniformly quasisymmetric group of homeomorphisms of $\boundary X$ can be conjugated by a quasisymmetric map to a group of homeomorphisms induced by an isometric action on $X$, provided that the set of conical limit points of the given action has positive measure.
If $X$ is a negatively curved symmetric space not equal to the real hyperbolic plane, the term ``quasiconformal'' may be substituted for ``quasisymmetric,'' since the two notions are equivalent in that setting
(see Heinonen--Koskela  \cite{HeinonenKoskela95} for the definition of quasiconformal and a proof of this equivalence).

The relevant statements about quasiconformal or quasisymmetric maps are proved for nondiscrete actions on $\Hyp_{\mathbb{R}}^2$ by Hinkkanen \cite{Hinkkanen90}, for discrete actions on $\Hyp_{\mathbb{R}}^2$ by Markovic \cite{Markovic06_Quasisymmetric}, for $\Hyp_{\mathbb{R}}^3$ by Sullivan and Tukia \cite{Sullivan81,Tukia80},  and for quaternionic and octave hyperbolic spaces by Pansu \cite{Pansu89}.  These results hold without any hypothesis on the set of conical limit points.

Analogous statements are shown for higher dimensional real hyperbolic spaces by Tukia \cite{Tukia86} and for complex hyperbolic spaces by Richard Chow \cite{Chow}, using the criterion in Corollary~\ref{cor:TukiaConicalLimit} as the definition of the term ``conical limit.''
See Appendix~\ref{sec:Conical} for the equivalence of this notion with the one given in Definition~\ref{def:ConicalLimit}.
\end{proof}

\begin{prop}
\label{prop:EquivariantQuasi}
Let $\GaP$ be relatively hyperbolic with a cusped space $Y$ that is quasi-isometric to a rank-one symmetric space $X$. Then the isometric action of $\Gamma$ on $Y$ is quasi-isometrically conjugate to an isometric action of $\Gamma$ on $X$.
\end{prop}

\begin{proof}
The isometric action of $\Gamma$ on $Y$ induces an action of $\Gamma$ by homeomorphisms on $\boundary Y$.
The quasi-isometry $f$ between $Y$ and $X$ conjugates the isometric action of $\Gamma$ on $Y$ to a quasi-action $\rho$ on $X$, which induces an action $\boundary \rho$ of $\Gamma$ by homeomorphisms on the boundary sphere $\boundary X$ by Theorem~\ref{thm:QItoQS}

The induced boundary homeomorphism $\boundary f$ topologically conjugates the action on $\boundary Y$ to the action on $\boundary X$.
Since the action of $\Gamma$ on $Y$ is cusp uniform, the induced action on $\boundary X$ is a geometrically finite convergence group action by Theorem~\ref{thm:ConvergenceGroup} and Theorem~\ref{thm:CuspUiffGeomFin}.
In particular, 
every point of $\boundary X$ is a conical limit point except for countably many parabolic points.
Thus we may apply Proposition~\ref{prop:Quasiconjugate}.
\end{proof}

\begin{proof}[Proof of Theorem~\ref{thm:RecognizingLattices}]
Assume that a cusped Cayley graph $Y$ for $(\Gamma,\mathbb{P})$ is quasi-isometric to a negatively curved symmetric space $X$. Since $X$ is $\delta$--hyperbolic, the space $Y$ is also Gromov hyperbolic.
Therefore, by Theorem~\ref{thm:grovesmanning} the pair $(\Gamma,\mathbb{P})$ is relatively hyperbolic.

By Proposition~\ref{prop:EquivariantQuasi}, the action of $\Gamma$ on $Y$ is quasi-isometrically conjugate to an isometric action on $X$.
In particular, the induced actions of $\Gamma$ on $\boundary Y$ and $\boundary X$ are topologically conjugate. Theorem~\ref{thm:CuspUiffGeomFin} states that the given action of $\Gamma$ on $\boundary Y$ is a geometrically finite convergence group action, so the action on $\boundary X$ shares these topological properties.

Because the action on $\boundary X$ is a discrete convergence group, it follows that $\Gamma$ acts properly on $X$ by Theorem~\ref{thm:CuspUiffGeomFin}.  Let $F \triangleleft \Gamma$ be the finite kernel of this action, so that $\bar\Gamma$ is isomorphic to a discrete subgroup of $\Isom(X)$.
Because the action on $\boundary X$ is geometrically finite, the quotient $\bar\Gamma \backslash X$ must have finite volume (see the implication (F2) $\Longrightarrow$ (F5) in \cite{Bowditch95}).
Thus $\bar\Gamma$ is isomorphic to a lattice, as desired.

Because the action of $\Gamma$ on $\boundary X$ is topologically conjugate to the action on $\boundary Y$, the set of group elements which stabilize exactly one boundary point is preserved. In particular, if $\phi$ is the $\Gamma$--equivariant homeomorphism $\boundary Y \rightarrow \boundary X$, then $\Stab_\Gamma (\xi) = \Stab_\Gamma\bigl(\phi(\xi)\bigr)$ for all bounded parabolic points $\xi \in X$. Then in the action by $\Gamma / F$ on $X$, the maximal parabolic subgroups will exactly be the images of these $\Stab_\Gamma\bigl(\phi(\xi)\bigr) = \mathbb{P}$ under the quotient map.
\end{proof}

\appendix
\section{Conical limit points: equivalence of definitions}
\label{sec:Conical}

This appendix discusses several definitions of conical limit point that are used throughout the literature, including interpretations in the settings of convergence groups, quasi-actions on $\delta$--hyperbolic spaces, and quasisymmetric groups acting on the boundary of a $\CAT(-1)$ space.
The definitions discussed here are widely known to be equivalent.  Yet we did not find a short, direct proof of their equivalence in the literature.

The main goal of this appendix is to prove Corollary~\ref{cor:TukiaConicalLimit}, which states that the definition of conical limit used by Bowditch in his work on relatively hyperbolic groups \cite{BowditchRelHyp} is equivalent to Tukia's criterion for conical limits used in the study of quasiconformal actions on the sphere at infinity of a negatively curved symmetric space \cite{Tukia86} and \cite{Chow}.
The equivalence established in Corollary~\ref{cor:TukiaConicalLimit} is used in Section~\ref{sec:SymmSpaces}.

We begin by recalling the definition of conical limit.

\begin{defn}[Conical limit]
\label{def:ConicalLimit:App}
Suppose $G$ has a convergence group action on $M$.  A point $\zeta \in M$ is a \emph{conical limit point} if there exist a sequence $(g_i)$ in $G$ and a pair of points $\xi_0\ne\xi_1 \in M$ such that
\[
   g_i \big| \bigl( M - \{\zeta\}\bigr) \to \xi_0
   \qquad \text{and} \qquad
   g_i(\zeta) \to \xi_1.
\]
\end{defn}


\begin{prop}
\label{prop:ConicalLimit}
Suppose $X$ is a proper $\delta$--hyperbolic space such that $\boundary X$ contains at least three points, and suppose $G$ quasi-acts on $X$.
A point $\zeta \in \boundary X$ is a conical limit point if and only if the following condition holds:

There exists a sequence $(g_i)$ in $G$ such that for some \textup{(}each\textup{)} $x \in X$ and some \textup{(}each\textup{)} geodesic ray $c\colon [0,\infty) \to X$ tending to $\zeta$, we have $g_i(x) \to \zeta$ and the distances $d \bigl( g_i(x),c \bigr)$ are uniformly bounded.
\end{prop}

The proof below derives from ideas of Beardon--Maskit and Bowditch \cite{BeardonMaskit74,Bowditch95,Bowditch99ConvergenceGroups}. See also \cite{Tukia98}.

\begin{proof}
First notice that if the property stated in the proposition holds for one choice of point $x \in X$ and ray $c$ then it clearly holds for any other choice.

Let $(g_i)$ be a sequence in $G$ such that for some $x \in X$ and for some ray $c$ asymptotic to $\zeta$, we have $g^{-1}_i(x) \to \zeta$  and the distances $d \bigl( g^{-1}_i(x),c\bigr)$ are uniformly bounded. 
Since $\boundary X$ has at least three points, there exist $\eta\ne\eta' \in \boundary X - \{\zeta\}$.
If $\ell$ is a geodesic line from $\eta$ to $\zeta$ then the distances $d \bigl( g^{-1}_i(x),\ell\bigr)$ also remain bounded. Thus the distances $d \bigl( x, g_i(\ell) \bigr)$ are bounded above, and the distance between the endpoints of each of the bi-infinite quasigeodesics $g_i(\ell)$ is bounded away from zero with respect to any visual metric.
Therefore the sequences $\bigl( g_i(\eta) \bigr)$ and $\bigl( g_i(\zeta) \bigr)$ do not have subsequences that converge to the same point.
Similarly $\bigl( g_i(\eta') \bigr)$ and $\bigl( g_i(\zeta) \bigr)$ do not have subsequences that converge to the same point.
By Theorem~\ref{thm:ConvergenceGroup} and compactness, we may pass to a collapsing subsequence such that
\[
   g_i \big| \bigl(\boundary X \setminus \{\zeta'\}\bigr) \to \xi_0
   \qquad \text{and} \qquad 
   g_i(\zeta') \to \xi_1
\]
for some points $\zeta',\xi_0,\xi_1 \in \boundary X$.
It follows immediately that $\zeta'$ must equal $\zeta$ and that $\xi_0 \ne \xi_1$.  Thus $\zeta$ is a conical limit point.

Conversely suppose $\zeta$ is a conical limit point.  Then $g_i \big| \bigl(\boundary X \setminus\{\zeta\}\bigr) \to \xi_0$ and $g_i(\zeta) \to \xi_1$ for some $\xi_0\ne \xi_1$.
Since $\boundary X$ has at least three points, there exist two distinct points $\eta$ and $\eta'$ in $\boundary X \setminus\{\zeta\}$. 
By hypothesis, the sequences $\bigl( g_i(\eta)\bigr)$ and $\bigl( g_i(\zeta) \bigr)$ have distinct limits.
Reversing the previous line of reasoning, we see that for each $x \in X$, if $\ell$ is a line from $\eta$ to $\zeta$ the distances $d\bigl( x, g_i(\ell) \bigr)$ and $d \bigl( g_i^{-1}(x),\ell \bigr)$ remain bounded as $i\to \infty$.
By Theorem~\ref{thm:ConvergenceGroup}, we have $g_i^{-1}(x) \to \infty$ since $(g_i^{-1})$ is a collapsing sequence.
Hence the set of limit points of $\bigl(g_i^{-1}(x) \bigr)$ is a nonempty subset of $\{\zeta,\eta\}$.
Similarly the set of limit points is also a nonempty subset of $\{\zeta,\eta'\}$.  Since $\eta \ne \eta'$, the only possibility is that $g_i^{-1}(x) \to \zeta$.
\end{proof}

In general, a convergence group action of $G$ on the boundary of a proper $\delta$--hyperbolic space $X$ might not extend to an isometric group action, or even a quasi-action, on $X$.  In that case, a natural substitute for $X$ is given by the \emph{triple space} of $\boundary X$, defined by
\[
   \Theta=\Theta(\boundary X) = \bigset{(\xi_1,\xi_2,\xi_3) \in (\boundary X)^3}{\text{$\xi_1,\xi_2,\xi_3$ are pairwise distinct}}.
\]
The group $G$ naturally acts on the triple space via the rule
\[
   g(\xi_1,\xi_2,\xi_3)=\bigl( g(\xi_1),g(\xi_2),g(\xi_3) \bigr).
\]

A correspondence between compact subsets of $\Theta(\boundary X)$ and compact subsets of $X$ may be described precisely in the special case that $X$ is a proper $\CAT(-1)$ space, using a construction due to Ahlfors--Cheeger.
We consider the function $\pi\colon \Theta(\boundary X) \to X$ such that $\pi(\xi_1,\xi_2,\xi_3)$ is the orthogonal projection of $\xi_3$ onto the unique geodesic of $X$ joining $\xi_1$ and $\xi_2$.  One can show that $\pi$ is a continuous  proper map (\emph{cf.}\ \cite[\S 8.2]{Gromov87} and \cite{Bowditch99ConvergenceGroups}), though we will not use this result.

Tukia and Chow have used the following criterion for conical limit points in their study of quasiconformal groups acting on the boundary of a negatively curved symmetric space in \cite{Tukia86,Chow}. 

\begin{cor}
\label{cor:TukiaConicalLimit}
Let $X$ be any proper visual $\CAT(-1)$ space whose boundary is uniformly perfect, and suppose $G$ acts as a uniformly quasisymmetric group on $\boundary X$.  A point $\zeta \in \boundary X$ is a conical limit point of the action if and only if the following condition holds:

There exists a sequence $(g_i)$ in $G$ such that for some \textup{(}each\textup{)} $(\xi_0,\xi_1,\xi_2)  \in \Theta$ and some \textup{(}each\textup{)} geodesic line $c$ tending to $\zeta$, we have
\[
   x_i = \pi g_i(\xi_0,\xi_1,\xi_2) \to \zeta
\]
and the distances $d ( x_i,c )$ are bounded above.
\end{cor}

\begin{proof}
The quasisymmetric action of $G$ on $\boundary X$ is the boundary action of a quasi-action of $G$ on $X$ by Theorem~\ref{thm:QStoQI}.  For any such quasi-action, the map $\pi$ must be roughly equivariant in the sense that the distances
\[
   d \bigl( \pi g_i(\xi_0,\xi_1,\xi_2),
     g_i \pi(\xi_0,\xi_1,\xi_2) \bigr)
\]
are uniformly bounded above.
Indeed, the point $\pi(\xi_0,\xi_1,\xi_2)$ lies within a uniformly bounded distance of all three sides of the corresponding ideal triangle.
Since the image of an ideal geodesic triangle under a quasi-isometry is within a bounded distance of an ideal geodesic triangle, the quasi-equivariance follows (\emph{cf.}\ \cite{Tukia85_QCExtensions} Thm.~3.6).
Thus the given condition is equivalent to the condition in Proposition~\ref{prop:ConicalLimit}.
\end{proof}

\bibliographystyle{alpha}
\bibliography{chruska.bib}

\end{document}